\documentclass[11pt]{amsart}

\usepackage{amsfonts, amssymb, amscd}

\usepackage{verbatim}
\usepackage{amssymb}
\usepackage{mathrsfs}
\usepackage{graphicx}
\usepackage{cite}
\usepackage[all]{xy}
\usepackage{tikz}
\usepackage[toc,page]{appendix}
\usepackage{tikz-cd}
\usepackage[margin=1.6in]{geometry}
\usepackage{graphicx}
\usepackage{float}

\newcommand{\even}{\mathrm{even}}
\newcommand{\odd}{\mathrm{odd}}

\newcommand{\Spec}{\mathrm{Spec}}
\newcommand{\Sym}{\mathrm{Sym}}

\newcommand{\corank}{\mathrm{corank}}
\newcommand{\rank}{\mathrm{rank}}

\newcommand{\Zz}{\mathbb{Z}}
\newcommand{\Cc}{\mathbb{C}}
\newcommand{\Pp}{\mathbb{P}}

\newcommand{\Qq}{\mathbb{Q}}

\newcommand{\Bb}{\mathcal{B}}

\newcommand{\Xx}{\mathcal{X}}

\newcommand{\Yy}{\mathcal{Y}}

\newtheorem{theorem}{Theorem}[section]
\newtheorem{lemma}[theorem]{Lemma}
\newtheorem{proposition}[theorem]{Proposition}
\newtheorem{definition}[theorem]{Definition}
\newtheorem{example}[theorem]{Example}
\newtheorem{corollary}[theorem]{Corollary}
\newtheorem{remark}[theorem]{Remark}

\numberwithin{equation}{section}

\begin{document}
\title[On stringy Euler characteristics of Clifford nc varieties]{On stringy Euler characteristics of Clifford non-commutative varieties}

\begin{abstract}
It was shown by Kuznetsov that complete intersections of $n$ generic quadrics in $\Pp^{2n-1}$ are related by Homological Projective Duality to certain non-commutative (Clifford) varieties which are in some sense birational to double covers of $\Pp^{n-1}$ ramified over symmetric determinantal hypersurfaces. Mirror symmetry predicts that the Hodge numbers of the complete intersections of quadrics must coincide with the appropriately defined Hodge numbers of these double covers. We observe that these numbers must be different from the well-known Batyrev's stringy Hodge numbers, else the equality fails already at the level of Euler characteristics. We define a natural modification of stringy Hodge numbers for the particular class of Clifford varieties, and prove the  corresponding equality of Euler characteristics in arbitrary dimension.
\end{abstract}

\author{Lev Borisov}
\address{Department of Mathematics\\
Rutgers University\\
Piscataway, NJ 08854} \email{borisov@math.rutgers.edu}

\author{Chengxi Wang}
\address{Department of Mathematics\\
Rutgers University\\
Piscataway, NJ 08854} \email{cw674@math.rutgers.edu}

\maketitle

\tableofcontents

\section{Introduction.}\label{section1}

Mirror symmetry is a relationship of duality between families of Calabi-Yau manifolds which has origins in superstring theory. There are two simplifications of the so-called nonlinear sigma models of superstring theory, the $A-$ and $B-$twists that depend only on the symplectic structure and the complex structure respectively, see for example \cite{Dcox}. A mirror symmetric pair is two Calabi-Yau manifolds $X$ and $Y$ which, together with some K\"ahler data,  give rise to two superconformal field theories that differ by a twist, see \cite{Dcox}.  

\medskip
The physical statement of being a mirror pair has multiple verifiable mathematical consequences. In particular, the $(p,q)-$ stringy Hodge numbers of $X$ and the $(\dim Y-p,q)-$stringy Hodge numbers of $Y$ are predicted to be equal. Also,  the Homological mirror symmetry conjecture proposed by Kontsevich in \cite{Mkon} predicts that the bounded derived category of coherent sheaves on $X$ and the derived Fukaya category of $Y$ with the appropriate symplectic structure are equivalent.

\medskip
There are several interesting examples when a single family of Calabi-Yau varieties $\{Y\}$ has two different mirror families $\{X_1\}$ and $\{X_2\}$. In this situation we refer to $\{X_1\}$ and $\{X_2\}$ as being double mirrors of each other, in the sense of them being mirror of a mirror. Such double mirror phenomenon predicts several mathematical statements that one may attempt to check rigorously in individual examples. Specifically, the complex moduli of $\{X_1\}$ and $\{X_2\}$ can be identified and the derived categories of the coherent sheaves for the corresponding elements of two families are expected to be equivalent. Some of these double mirrors are simply birational Calabi-Yau varieties, as is the case in the toric setting of \cite{Batyrev.2, LBo, ZhanLi, Clarke, FaveroKelly}. More interesting examples appear in the setting of nonabelian gauged linear sigma models and/or Homological Projective Duality of \cite{HPD}. For instance, the Pfaffian-Grassmannian correspondence of \cite{Rodland} has been verified in \cite{BorCald, Kuz1} to lead to derived equivalent non-birational Calabi-Yau varieties. It has also been used to construct zero divisors in the Grothendieck ring of algebraic varieties, see \cite{BorK0}.

\medskip
One of the early examples of Homological Projective Duality is that of complete intersections of quadrics, and it is the main focus of our paper.
Specifically, let $V$ be a complex  vector space of dimension $2n$.  In \cite{ku1,ku2}, Kuznetsov considered a Lefschetz decomposition of the derived category of $\mathbb{P}(V)$ with respect to the double Veronese embedding $\mathbb{P}(V)\rightarrow \mathbb{P}(\Sym^2V)$. He verified that the non-commutative algebraic variety $(\mathbb{P}(\Sym^2V^*),\mathcal{B}_0)$, where $\mathcal{B}_0$ is the universal sheaf of even parts of Clifford algebras over $\mathbb{P}(\Sym^2V^*)$, is Homologically Projectively Dual to $\mathbb{P}(V)$ with respect to this Lefschetz decomposition.  As a consequence, for a generic vector subspace $W\subset \Sym^2V^*$ of dimension $n$ there is  an equivalence 
\begin{equation}\label{kuzder}
\mathcal{D}^b(Y_W)\cong \mathcal{D}^b(\mathbb{P}W,\mathcal{B}_0)
\end{equation}
where $Y_W\subset \Pp V$ is the complete intersection of quadrics defined by $W$ and $\mathcal{D}^b(\mathbb{P}W,\mathcal{B}_0)$ is the derived category of sheaves of $\mathcal{B}_0-$modules on $\mathbb{P}W$. We call the sheaves of algebras $\mathcal{B}_0$ on $\mathbb{P}W$ \emph{Clifford non-commutative} varieties.

\medskip
We interpret the result \eqref{kuzder} of \cite{ku1} as the statement that $(\mathbb{P}W,\mathcal{B}_0)$ and $Y_W$ are double mirror to each other.
In particular, it is natural to try to consider some kind of Euler characteristics of the non-commutative Clifford variety $(\mathbb{P}W,\mathcal{B}_0)$ and to  predict it to be the same as the Euler characteristics of the complete intersection $Y_W$. 

\medskip 
To a Clifford non-commutative variety $(\mathbb{P}W,\mathcal{B}_0)$ we can associate the relative $\Spec$ of the center of $\mathcal{B}_0$ which gives a double cover $Z=Z_W\to\Pp W$. The sheaf of algebras $\mathcal{B}_0$ is almost a sheaf of Azumaya algebras over $Z$ in the sense that it is so over a Zariski open subset. Note that $Z$ is singular for $n>3$ and we can view $(\mathbb{P}W,\mathcal{B}_0)$ as a non-commutative crepant resolution of singularities of $Z$ (equipped with a Brauer class to account for the Azumaya algebras). 

\medskip
The first reasonable approach to defining Hodge numbers of  $(\mathbb{P}W,\mathcal{B}_0)$  is to use the Batyrev's stringy Hodge numbers of $Z$, see \cite{Batyrev.1}. These are defined for varieties with log-terminal singularities, and $Z$ has, in fact, Gorenstein canonical singularities.
Unfortunately, this approach fails already at the level of Euler characteristics, in the sense that 
$$
\chi_{st}(Z) \neq \chi(Y_W)
$$
for large enough $n$. Moreover, Batyrev's stringy Euler characteristics of $\chi_{st}(Z)$ can be non-integer. Remarkably, a simple change of discrepancies (Definition \ref{aij}) allows us to define a Batyrev-style invariant (Definition \ref{defcli-stri}) which we call Clifford-stringy Euler characteristics that rectifies this.  Namely, we are able to verify our prediction and get equality between the Clifford-stringy Euler characteristics $\chi_{cst}(Z)$ and the Euler characteristics of the complete intersection, which is our main result Theorem \ref{main}. 

\bigskip

{\bf Theorem \ref{main}.}
Let $Y=\{v\in V| q(v)=0$ for all $q\in \mathbb{P}W\}$ be the complete intersection associated to $W$ in $\mathbb{P}V $. Then we have $$\chi_{cst}(Z)=\chi(Y).$$

\medskip

The paper is organized as follows.
In Section \ref{section2}, we first recall the definition of complete quadrics and their moduli space  $\widehat{\mathbb{P}W}$, which are important technical tools for constructing resolutions of singularities of $Z$. This allows us to construct a partial resolution of singularities $\theta:\widehat{Z}\rightarrow Z$ where $\widehat{Z}$ has explicitly written toroidal singularities. We are able to resolve these singularities of $\widehat{Z}$ in a systematic way by using the combinatorial construction of  Appendix \ref{res} to get the toroidal resolution $\mu:\widetilde{Z}\rightarrow\widehat{Z}$.
Section \ref{cli-str} contains the calculation for the discrepancies of exceptional divisors of the resolution $\theta\circ\mu:\widetilde{Z} \rightarrow Z$. Then we define the modified discrepancies of these exceptional divisors and the Clifford-stringy Euler characteristics of $Z$.
Section \ref{fiber} computes Euler characteristics of strata of fibers of $\theta$.
Section \ref{contribu} begins to calculate the Clifford-stringy Euler characteristics of $Z$ by analyzing the local contributions of singular points. We deduce a combinatorial expression of the formula of the local contribution in Proposition \ref{chit} and calculate several examples to find an interesting $1,2,1,2, \ldots$ pattern.
Section \ref{sec.hard}, which is the most cumbersome part of the paper, gives the proof for this $1,2,1,2, \ldots$ pattern.
In Section \ref{same}, we prove our main Theorem \ref{main} which reveals the relationship between the Clifford-stringy Euler characteristics and the Euler characteristics of the complete intersection of quadrics.
Section \ref{open} contains several comments and open questions.
In Appendix \ref{res}, we describe the resolution needed in Section \ref{section2}. We also verify a certain convexity property for the triangulation.

\medskip
\noindent{\it Acknowledgements.} L.B. has been partially supported by NSF grant DMS-1601907. 

\section{Complete quadrics, the double cover and resolution of singularities.}\label{section2}
In this section, we first recall the definition of the complete quadrics. Then we construct two double covers $\sigma: Z\rightarrow \mathbb{P}W$, $\rho: \widehat{Z}\rightarrow \widehat{\mathbb{P}W}$ and a map $\theta:\widehat{Z} \rightarrow Z$. We show that the singular points on $\widehat{Z}$ are all toroidal. Then we resolve the singularities of $\widehat{Z}$ by toric methods of  Appendix \ref{res}.

\subsection{Complete quadrics.}\label{complete}
Let $V$ be a vector space of dimension $n$. Let $\Phi= \mathbb{P}\Sym^2 V^*$ be the space of quadratic forms on V up to scaling. For arbitrary $\mathbb{C}q\in \mathbb{P}\Sym^2 V^*$,
if $$q=f(x_1,\ldots, x_n)=\Sigma_{1\leq i \leq j \leq n}a_{ij}x_ix_j $$
in a basis $\{x_i | 1\leq i \leq n\}$ of $V^*$, then the corresponding symmetric matrix is

$$
   M(q)=
   \begin{pmatrix}
   a_{11} & \frac{1}{2}a_{12} & \frac{1}{2}a_{13} & \ldots \\
   \frac{1}{2}a_{12} & a_{22} & \frac{1}{2}a_{23} & \ldots \\
   \frac{1}{2}a_{13} & \frac{1}{2}a_{23} & a_{33} & \ldots \\
   \vdots & \vdots & \vdots & \ddots
  \end{pmatrix}
$$
and $Rad(q)=\{v\in V | M(q)v=0 \}$ is a subspace of $V$, which is called the radical of $q$. Quadratic form $q$ is called non-degenerate if $Rad(q)=0$. When $Rad(q)\neq 0$, $q$ is degenerate as a
quadratic form on $V$. However, $q$ can be regarded as a non-degenerate quadratic form on $V/Rad(q)$. Let the quadric corresponding to $q$ be $Q$ which is zero set of $q=f(x_1,\ldots x_n)$ in $\mathbb{P}V$. We define the rank of $q$ or $Q$ to be the rank of $M(q)$, and the corank as the dimension of $Rad(q)$. The singular locus of $Q$ is $\mathbb{P}Rad(q)$.

\medskip

\begin{definition}\cite{Thaddeus}
A complete quadric in $\mathbb{P}W$ is a finite sequence of quadrics $Q_i$, where $Q_1$ is a hypersurface in $\mathbb{P}V$, $Q_{i+1}$ is a hypersurface in the singular locus of $Q_i$, and the last $Q_i$ is smooth.
\end{definition}
Let $q_i$ be the symmetric form corresponding to $Q_i$, $V_0=V$ and $V_i=Rad(q_i)$ when $i>1$. Then $q_i$ can be regarded as a non-degenerate quadratic form on $V_{i-1}/V_i$. Thus, we have the following equivalent definition of a complete quadric.

\begin{definition}\label{decrease}
A complete quadric is a flag $0= V_l \subset V_{l-1} \subset \ldots \subset V_1 \subset V_0=V$ together with non-degenerate quadratic forms $\mathbb{C}q_i \in \mathbb{P}\Sym^2(V_{i-1}/V_i)^*$.
\end{definition}
Let $\Phi$ be the space of quadrics in $\mathbb{P}V$, which equals the projective space $\mathbb{P}\Sym^2 V^*$. Let $\Phi_1=\{Q \in \Phi | \corank\, Q \geq 1\}$ be the subvariety of $\Phi$ consisting of all the quadrics with corank $\geq 1$. In other words, the singular quadrics form a hypersurface $\Phi_1$ of degree $n$, given by the vanishing of the determinant of the matrix $M(q)$ above. The hypersurface $\Phi_1 \subseteq \mathbb{P}\Sym^2 V^*$ is not smooth. Its singular locus $\Phi_2$ equals $\{Q \in \Phi | \corank\, Q \geq 2\}$, which has codimension $3$ in $\Phi$. In general, if $\Phi_k=\{Q \in \Phi | \corank\, Q \geq k\}$, $\Phi_k$ has codimension
$
   \begin{pmatrix}
k+1 \\
2
\end{pmatrix}
$
in $\Phi$. When $k < n-1$, $\Phi_k$ is not smooth and its singular locus is $\Phi_{k+1}$. The locus $\Phi_{n-1}=\{Q \in \Phi | \corank\, Q \geq n-1\}=\{Q \in \Phi | \rank\, Q = 1\}$ is $\mathbb{P}V^*$ in its veronese embedding in $\mathbb{P}\Sym^2 V^*$. These loci $\Phi_i$ form a stratification of $\Phi$. The space of complete quadrics provides a log resolution of this stratification. It is described in the following proposition.

\begin{proposition}
Consider the successive blowups along $\Phi_{n-1}$, the loci of $quadrics$ of corank $\geq n-1$ in $\Phi$, then the proper preimage of $\Phi_{n-2}$, the loci of quadrics of corank $\geq n-2$ in $\Phi$, etc. At each stage the center of the blowup is smooth, so all of the blowups are smooth. The resulting space is the space of complete quadrics, which  is a smooth variety $\widehat{\Phi}$ which parametrizes the flags $$0= F^0 \subset F^1 \subset \ldots \subset F^{l-1} \subset F^l=V$$ together with non-degenerate quadrics $\mathbb{C}q_i \in \mathbb{P}\Sym^2(F^{i+1}/F^i)^*$. The map $ \widehat{\Phi} \rightarrow \Phi$ which we denote by $\pi$ is given by interpreting a quadratic from on $F^l/F^{l-1}$ as a quadratic form on $F^l=V$.
\end{proposition}
\begin{proof}
See \cite{Bertram, Thaddeus}.
\end{proof}

\medskip

\begin{remark}
Note that we now use increasing filtrations in contrast with Definition \ref{decrease}. We have $F^i=V_{l-i}$, $0\leq i \leq l$.
\end{remark}

\medskip

\begin{proposition}\label{2.5}
When $2\leq i \leq n-1$, let $D_i$ be the exceptional divisors of the map $\pi: \widehat{\Phi} \rightarrow \Phi$ corresponding to $\Phi_i$; let $D_1$ be the proper preimage of $\Phi_1$ under the map $\pi$. Then $D_i$ are described by requiring that a subspace of dimension $i$ is present in the flag $F^\bullet$, that is to say $D_i=\{flag \,F^\bullet \in \widehat{\Phi} | $ a subspace of dimension $i$ is present in the flag $F^\bullet\}$. These $D_i$, $1\leq i \leq n-1$ form a simple normal crossing divisor on $\widehat{\Phi}$. The generic point of $D_i$ maps by $\pi$ to the generic point of $\Phi_i$. The discrepancy of $D_i$ is codimension $(\Phi_i)-1=\frac{i(i+1)}{2}-1$, $2\leq i \leq n-1$.
\end{proposition}
\begin{proof}
From the iterated blowup construction we see that the exceptional divisor is a simple normal crossing divisor (see \cite{Bertram}). The description of it at the set-theoretic level is clear as well. The discrepancies are calculated based on the codimension of the (smooth) locus of the corresponding blowup.
\end{proof}

\subsection{Double cover.}\label{Doub}
Let $V$ be a vector space of dimension $2n$ and $\Sym^2 V^*$ be the space of symmetric bilinear forms on $V$ which we will also identify with the space of quadratic forms on $V$. We choose a dimension $n$ subspace $W\subset \Sym^2 V^*$ in general position and consider $\mathbb{P}W\subseteq \Phi=\mathbb{P}\Sym^2 V^*$. We denote by $\phi_i=\Phi_i \cap \mathbb{P}W$ the loci of forms in $\mathbb{P}W$ of corank at least $i$. In particular, $\phi_1$ is a degree $2n$ hypersurface in $\Pp W$. Indeed, if we pick bases $(q_1,\ldots,q_n)$ of $V$ and $(x_1,\ldots,x_{2n})$ of $W$, then the quadratic form 
$$
q=\sum_{i=k}^n u_k q_k\overline{x}\in W
$$
is degenerate if and only if the determinant of the $2n\times 2n$ matrix of linear forms on $\Pp W$ 
$$
\det\left(\frac{\partial ^2}{\partial x_i \partial x_j}\sum_1^nu_kq_k(\overline{x})\right)_{i,j}
$$
is zero.

\medskip
We now consider the double cover of $\Pp W$ ramified in $\phi_1$. It can be written in coordinates as a hypersurface in a weighted projective space
$$
y^2=\det\left(\frac{\partial ^2}{\partial x_i \partial x_j}\sum_1^nu_kq_k(\overline{x})\right)_{i,j}
$$
where $y$ has weight $n$ and $u_i$ have weight one. We denote this double cover by $Z$. Its significance stems from the work of Kuznetsov which identifies 
$Z$ with (the relative $\Spec$ of)  the center of a certain Clifford non-commutative variety.
We observe that the singular locus of $Z$ is the (preimage of) the singular locus of the ramification divisor, which is $\sigma^{-1}(\phi_2)$.

\medskip

We choose an integer $k$ such that $\frac{(k+1)(k+2)}{2} > n-1$ and $\frac{k(k+1)}{2} \leq n-1$. Since $\Phi_i$ has codimension $\frac{i(i+1)}{2}$ in $\Phi$
and the dimension $n$ subspace $W$ is generic, we have $\phi_{k+1}=\mathbb{P}W \cap \Phi_{k+1} = \emptyset$  and $\phi_i=\mathbb{P}W \cap \Phi_i \neq \emptyset$ when $i=1,2, \ldots , k$. Now we consider the restriction of map $\pi$ to $\pi^{-1}(\mathbb{P}W)$, which we also denote by $\pi: \widehat{\mathbb{P}W} \rightarrow \mathbb{P}W$, where $\widehat{\mathbb{P}W}=\pi^{-1}(\mathbb{P}W)$. Let $\phi_i=\Phi_i \cap \mathbb{P}W$ and $T_i=D_i \cap \widehat{\mathbb{P}W}$, $i=1,2, \ldots , k$. Since $\pi(D_i)=\Phi$, we have $\pi(E_i)=\phi_i$. Actually, $\pi: \widehat{\mathbb{P}W} \rightarrow \mathbb{P}W$ is a log-resolution of $\mathbb{P}W$, i.e. a proper birational morphism from a smooth variety $\widehat{\mathbb{P}W}$ such that the exceptional divisor $\bigcup_{i=1}^{k}T_i$ has simple normal crossings. Because $W$ is transversal to all $\Phi_i$, we have $$K_{\widehat{\mathbb{P}W}}=\pi^*K_{\mathbb{P}W}+\sum_{i=1}^{k}(\frac{i(i+1)}{2}-1)T_i.$$ By \cite{Bertram}, we have
\begin{equation}\label{iti}
\pi^*\phi_1=\sum_{i=1}^kiT_i
\end{equation}
which will be important for our calculations.

\medskip
Denote by $\widehat{Z}$ the normalization of the fiber product of $\pi: \widehat{\mathbb{P}W} \rightarrow \mathbb{P}W$ and $\sigma: Z \rightarrow \mathbb{P}W$. Then we have the following commutative digram.

\hskip 50pt \xymatrix{
\widehat{Z} \ar[d]_{\rho} \ar[r]^{\theta}
& Z \ar[d]_{\sigma}\\
\widehat{\mathbb PW} \ar[r]^{\pi}
& \mathbb PW
}

\noindent
We have $\rho: \widehat{Z} \rightarrow \widehat{\mathbb PW}$ is a double cover ramified over the $T_i$ for all odd $i$ in the range $1\leq i \leq k$. Indeed, the normalization only occurs over the $T_i$ which come with odd coefficients in $\pi^*\phi_1$ in \eqref{iti}.
 \footnote{This is analogous to the fact that normalization of $\Zz[\sqrt{m}]$ equals $\Zz[\sqrt{m'}]$, where $m'$ is the square free part of integer $m$.} Let $E_i$ be the proper transformation of $T_i$ under the map $\rho$ and $\phi_{iZ}$ be the preimage of $\phi_i$ under the map $\sigma$, where $1\leq i \leq k$. Then $\theta(E_i)=\phi_{iZ}$.

\medskip
\subsection{Resolution of singularities.}\label{resoZ}
Singularities of $\widehat{Z}$ are induced from singularities of the ramification divisor $\sum_{1\leq i \leq k,~i={\rm odd}}T_i$ which has simple normal crossings. 
Specifically, given a subset $S\subseteq \{1,3,5,\dots,2l-1\}$ , where $l=[\frac{k+1}{2}]$, we have for a point $p$ in $(\cap_{i\in S}T_i)\backslash\cup_{j\notin S}T_j$
an analytic neighborhood of $p$ where $T_i, i\in S$ are given by local coordinates $t_i$. Then a local analytic neighborhood of $p'=\rho^{-1}(p)$ is given by $y^2=\prod_{i\in S}t_{i}$ times a disc of dimension $(n-1)-|S|$. Thus singularities of $\widehat{Z}$ are toroidal and are amenable to usual toric resolution techniques, see \cite{Kempf}. We consider the corresponding toric singularities in Appendix \ref{res} where we provide an explicit resolution.

 \medskip
Specifically, in Appendix \ref{res}, we construct the resolution $\mu_{toric}: \widetilde{Z}_{toric} \rightarrow \widehat{Z}_{toric}$ which is a local description of the resolution of singularities of $\widehat{Z}$. In particular, we introduce additional exceptional divisors $\widetilde{E}_{2i-1,2j-1,toric}$  corresponding to the rays $e_{i,j}=\frac{e_i+e_j}{2}$ in $\Sigma$, where $i,j\in \{1,2,\dots,l\}$ and $i< j$. By \cite{Kempf}, we obtain the resolution of $\widehat{Z}$, $\mu: \widetilde{Z} \rightarrow \widehat{Z}$ and the corresponding exceptional divisors of the global resolution $\mu$ which we denote by $\widetilde{E}_{2i-1,2j-1}$, where $i,j\in \{1,2,\dots,l\}$ and $i< j$.
 Also, the proper transformations of the divisor $E_s$ under the resolution $\mu: \widetilde{Z}\to \widehat{Z}$ are denoted by $\widetilde{E}_s$ where $s=1,2,\dots,k$. Thus, the exceptional divisors of the resolution of singularities of $Z$, $\theta\circ\mu:\widetilde{Z}\to \widehat{Z}\to Z$ are $\widetilde{E}_s$, where $s=1, 2,\ldots,k$ and $\widetilde{E}_{s,t}$, where $s,t \in \{1,3,\ldots, 2l-1\}$ with $l=[\frac{k+1}{2}]$.

\section{Clifford-stringy Euler characteristics.}\label{cli-str}
This section contains the calculation for the discrepancies of exceptional divisors of the resolution $\theta\circ\mu:\widetilde{Z} \rightarrow Z$. Then we will define the modified discrepancies of those exceptional divisors and the Clifford-stringy Euler characteristics of $Z$.
\subsection{Discrepancies.}\label{Discre}

We have obtained the following commutative diagram in Section \ref{Doub}.
\hskip 100 pt
\xymatrix{
\widehat{Z} \ar[d]_{\rho} \ar[r]^{\theta}
& Z \ar[d]_{\sigma}\\
\widehat{\mathbb PW} \ar[r]^{\pi}
& \mathbb PW
}

We denote by  $T_i$  the proper preimages of $\phi_i\subset\Pp W$ under $\pi$ and by $\phi_{i,Z}\subseteq Z$  the  preimages of $\phi_i$ under $\sigma$. 
We denote by $E_i$ be the proper preimages of $T_i$ under $\rho$ for  $1\leq i \leq k$. Then we have the following lemma.

\begin{lemma}\label{discre}
There holds $\theta^*\phi_{1,Z}=\sum_{i=1}^{k}\beta_i E_i$ and $K_{\widehat{Z}}=\theta^*K_Z+\sum_{i=1}^{k}\gamma_iE_i$ with  
\begin{equation}\label{beta}
  \beta_i=\begin{cases}
    i, & \text{if $i\text{ is odd}$};\\
    \frac{i}{2}, & \text{if $i\text{ is even}$},
  \end{cases}
\end{equation}
\begin{equation}\label{gamma0}
  \gamma_i=\begin{cases}
    i^2-1, & \text{if $i\text{ is odd}$};\\
    \frac{i^2}{2}-1, & \text{if $i\text{ is even}$},
  \end{cases}
\end{equation}
in the sense of rational equivalence of $\Qq$-Carter divisors.
\end{lemma}

\begin{proof}
From the commutative diagram, we have $\theta^*\sigma^*\phi_1=\rho^*\pi^*\phi_1$ and we compare the two sides.
Since $\sigma$ is a double cover ramified at $\phi_{1,Z}$, we get $\sigma^*\phi_1=2\phi_{1,Z}$, so the left hand side equals
$2 \theta^* \phi_{1,Z}$. To calculate the right hand side, we use 
\begin{equation}\label{eqB}
\pi^*\phi_1=\sum_{i=1}^{k}iT_i
\end{equation}
from \cite{Bertram}. We also use
\begin{equation}\label{rami}
 \rho^*T_i=\begin{cases}
    2E_i, & \text{if $i\text{ is odd}$};\\
    E_i, & \text{if $i\text{ is even}$},
  \end{cases}
\end{equation}
due to the fact that $\rho$ is a double cover which is ramified over $\{T_i| $ i is  odd$\}$.
Thus
$$
2\theta^*\phi_{1,Z}=\rho^*\pi^*\phi_1 \stackrel{\eqref{eqB}}=\sum_{i=1}^k i \rho^* T_i
 \stackrel{\eqref{rami}}=\sum_{i\text{ is even,}1\le i\le k}iE_i+\sum_{i\text{ is odd,}1\le i\le k}2iE_i
$$
which implies \eqref{beta}.

\medskip

Similarly, the commutative diagram implies
\begin{equation}\label{commu}
\theta^*\sigma^*K_{\mathbb PW}=\rho^*\pi^*K_{\mathbb PW}.
\end{equation}
Since $\sigma$ is the double cover ramified at $\phi_1$, the left hand side of \eqref{commu} equals  
\begin{equation}\label{left}
\theta^*\sigma^*K_{\mathbb PW}=\theta^*(K_Z-\phi_{1,Z}) = \theta^*K_Z - \sum_{i=1}^k \beta_i E_i.
\end{equation}
We use the discrepancies formula of Proposition \ref{2.5} to see that the right hand side equals
\begin{align*}\label{right}
\rho^*\pi^*K_{\mathbb PW} &=\rho^*(K_{\widehat{\mathbb PW}} - \sum_{i=1}^{k}(\frac{i(i+1)}{2}-1)T_i) 
\\&=\rho^*K_{\widehat{\mathbb PW}} -\sum_{i\text{ is odd,}1\le i\le k}(i(i+1)-2)E_i-\sum_{i\text{ is even,}1\le i\le k}(\frac{i(i+1)}{2}-1)E_i
\\&= K_{\widehat{Z}} -\sum_{i\text{ is odd,}1\le i\le k}(i(i+1)-1)E_i-\sum_{i\text{ is even,}1\le i\le k}(\frac{i(i+1)}{2}-1)E_i
\end{align*}
where in the last equality we used that $\rho$ is ramified at $E_i$ for odd $i$. It remains to combine this with \eqref{commu} and  \eqref{left}
and use \eqref{beta} to get \eqref{gamma0}.
\end{proof}

Now we can calculate the discrepancies of the exceptional divisors of $\widetilde Z\to Z$. Recall that 
by construction of Section \ref{Doub} and \ref{resoZ}, we have the following commutative diagram:
$$\xymatrix{
  \widetilde{Z} \ar@/_/[ddr]_{\omega'} \ar@/^/[drr]^{\omega}
    \ar@{>}[dr]|-{\mu}                   \\
   & \widehat{Z} \ar[d]^{\rho} \ar[r]_{\theta}
                      & Z \ar[d]_{\sigma}    \\
   & \widehat{\mathbb{P}W} \ar[r]^{\pi}     & \mathbb{P}W     }$$

\begin{proposition}\label{discrep}
There holds
\begin{align*}
K_{\widetilde{Z}}=&\mu^*\theta^*K_Z+\sum_{s\text{ is even,}1\le s\le k}(\frac{s^2}{2}-1)\widetilde{E}_s + \sum_{s \in \{1,3,\ldots, 2l-1\}}(s^2-1)\widetilde{E}_s\\
&+\sum_{\{s,t\} \subseteq \{1,3,\ldots, 2l-1\}}(\frac{s^2+t^2}{2}-1)\widetilde{E}_{st}
\end{align*}
in the sense of equivalence of $\Qq$-Carter divisors,
where $K_{\widetilde{Z}}$ and $K_Z$ are canonical divisors of $\widetilde{Z}$ and $Z$ respectively.
\end{proposition}

\begin{proof}
The map $\mu :\widetilde{Z}\to  \widehat{Z}$ is a toroidal morphism, with the corresponding toric geometry described 
In Appendix \ref{res}. We see that $\mu$ is crepant, so $\mu^*K_{\widehat Z} = K_{\widetilde Z}$. We also observe that the structure in codimension one is
that of introducing new divisors $E_{st}$ in the blowup of surface $A_1$ singularities along a disc. This means that for odd $s$ we have 
\begin{equation}\label{odd}
\mu^*E_s=\widetilde{E}_s+\frac{1}{2}\sum_{t \in \{1,3,\ldots, 2l-1\} \backslash \{s\}}\widetilde{E}_{st}.
\end{equation}
Together with $\mu^*E_s=\widetilde{E}_s$ for even $s$ and equation \eqref{gamma0} of Lemma \ref{discre}, we get the desired result.
\end{proof}

\subsection{Definition of modified discrepancies and Clifford-stringy Euler characteristics.}\label{mo-dis-cli}
Let $X$ be a singular variety with log-terminal singularities. Let $\pi:\widehat X \to X$ be a log resolution of $X$, i.e. a proper birational morphism from a smooth variety $\widehat X$ such that the exceptional locus is a divisor $\bigcup_{i=1}^kD_i$ with simple normal crossings. It is assumed that $X$ is $\mathbb{Q}$-Gorenstein,
which allows us to compare the canonical classes
$$
K_{\widehat X}\equiv \pi^*K_X + \sum_{i=1}^{k}\alpha_iD_{i}
$$
to define the discrepancies $\alpha_i$. The discrepancies satisfy $\alpha_i>-1$ by log-terminality assumption (see \cite{CKM}).
Recall that for any variety  $W$ (not necessarily projective) we can define the Hodge-Deligne polynomial $E(W;u,v)$ which
measures the alternating sum of dimensions
of the $(p,q)$ components of the mixed Hodge structure on the cohomology
of $W$ with compact support \cite{DKh}.
For a possibly empty subset $J$ of $\{1,\ldots,k\}$ we define by
$D_J^\circ$ the locally closed subset of $\widehat X$ which consists of points $z\in \widehat X$ that lie in $D_j$ if and only
if $j\in J$.

\medskip
Motivated by mirror symmetry calculations, Batyrev in \cite{Batyrev.1} has defined the following invariant of singular varieties with log-terminal singularities.
\begin{definition}\label{def.Est}
Stringy $E$-function of a variety $X$ with log-terminal singularities is defined by
$$
E_{st}(X;u,v):=\sum_{J\subseteq \{1,\ldots , k\}} E(D_J^\circ;u,v) \prod_{j\in J} \frac {uv-1}{(uv)^{\alpha_j+1}-1}.
$$
\end{definition}

\begin{remark}
Thus defined $E_{st}(X;u,v)$ does not depend on the choice of the log resolution $\widehat X$ of $X$, which justifies the notation.
It is in general only a rational function in fractional powers of $u,v$ but is polynomial in many cases of interest. Various specializations of $E_{st}(X;u,v)$ include the stringy $\chi_y$ genus $E_{st}(X;y,1)$ and the stringy Euler characteristics 
$$
\chi_{st}(X)=\lim_{(u,v)\to(1,1)}E_{st}(X;u,v)=\sum_{J\subseteq \{1,\ldots , k\}} \chi(D_J^\circ) \prod_{j\in J} \frac 1{\alpha_j+1}
$$
where $\chi$ denotes the Euler characteristics with compact support.
\end{remark}

\medskip
In a number of cases,  the log resolution of singularities $\widehat X\to X$ has an additional property of having the open strata $D_J^\circ$  form Zariski locally trivial fibrations over the corresponding
strata in $X$.  This is the case when $X$ has isolated singularities, but it also occurs more generally.
Notably, this  happens in the case of generic hypersurfaces and complete intersections in toric
varieties,  as well as in the case considered in this paper. We discuss this phenomenon below.

\begin{definition}\label{zariski}
We call a log resolution $\pi:\widehat X\to X$ as above \emph{Zariski locally trivial} if each $D_J^\circ$ is a Zariski locally trivial fibration over its image $\pi(D_J^\circ)$ in $X$.
\end{definition}

\begin{definition}\label{s}
Suppose that a singular variety $X$ admits a Zariski locally trivial log resolution
$\pi:\widehat X\to X$. For a point $y\in X$ define the local contribution of $y$ to $E_{st}(X;u,v)$ to
be
$$
S(y;u,v):=\sum_{J\subseteq \{1,\ldots , k\}} E(D_J^\circ\cap \pi^{-1}(y);u,v) \prod_{j\in J} \frac {uv-1}{(uv)^{\alpha_j+1}-1}.
$$
\end{definition}

\begin{remark}
Thus defined $S(y;u,v)$ is a constructible function on $X$ with values in the field of rational functions in fractional powers of $u$ and $v$. Indeed, if $y_1$ and $y_2$ are such that the set of $J$ with $\pi(D_J^{\circ})$ that contain $y_1$ is the same as those that contain $y_2$, then $S(y_1;u,v)=S(y_2;u,v)$.
Moreover, this function is independent from the choice of a Zariski locally trivial resolution. This follows from
the usual argument that involves weak factorization theorem \cite{AKMW}. Last but not least, there holds
\begin{equation}\label{Zarlogtriv}
E_{st}(X;u,v) = \sum_{i} E(X_i;u,v) S(y\in X_i;u,v),
\end{equation}
where $X=\bigsqcup_i X_i$ is the stratification of $X$ into the sets on which $S$ is constant. This follows immediately
from the multiplicativity of Hodge-Deligne $E$-functions for Zariski locally trivial fibrations, see \cite{DKh}.
\end{remark}

Recall that given a generic subspace $W\subseteq \Sym^2 V^*$ with $\dim W=n$ and $\dim V=2n$ Kuznetsov 
in \cite{ku1} constructs two varieties, the complete intersection of $n$ quadrics 
$$
Y_W\subseteq \Pp V =\{\Cc v\in \Pp V, \mbox{ such that} ~q(v,v)=0,~\mbox{for all } q\in W\}
$$
and a certain sheaf of Clifford algebras $\Bb_0$ on $\Pp W$ which can be generically viewed as a sheaf of Azumaya algebras over the double cover $Z=Z_W$ of $\Pp W$ considered in Section \ref{section2}. The equivalence of derived categories makes it plausible to expect that 
$$
E_{st}(Z;u,v) = E(Y_W;u,v),
$$
and in particular $\chi_{st}(Z)=\chi(Y_W)$. However, this is not the case for large enough $n$ (specifically, the stringy Euler characteristics statement first fails at $n=7$).

\bigskip
The main idea of this paper is that one can salvage the expected equality by a modification of the definition of the stringy Euler characteristics of $Z$, which we call \emph{Clifford-stringy} Euler characteristics of $Z$.
Recall that we have a resolution 
$$
\widehat Z\to Z
$$
with components of the exceptional divisor $\widetilde E$ indexed by the set $S=\{i|1\leq i\leq k\} \cup \{(i,j)|1\leq i < j \leq k$ and $i, j$ are odd$\}$. For a subset $J$  of $S$, we define us usual
\[\begin{split}
&\widetilde{E}_J=\bigcap_{\diamondsuit \in J}\widetilde{E}_{\diamondsuit}\text{, } \widetilde{E}_{\emptyset}=\widetilde{Z},\\
&\widetilde{E}_J^{\circ}=\bigcap_{\diamondsuit \in J}\widetilde{E}_{\diamondsuit}-
\bigcup_{\diamondsuit \in S-J}\widetilde{E}_{\diamondsuit}.
\end{split}\]

\begin{definition}\label{aij}
For $\omega=\theta\circ\mu: \widetilde{Z} \rightarrow Z$,
the modified discrepancies of the divisors $\widetilde{E}_{\diamondsuit}$ are defined by
\begin{equation}\label{gamma}
  a_{\diamondsuit}=\begin{cases}
    i^2-1, & \text{if $\diamondsuit=i, 1\leq i\leq k$, $i\text{ is odd}$};\\
    \frac{i^2}{2}-1, & \text{if $\diamondsuit=i, 1\leq i\leq k$,$i\text{ is even}$};\\
    \frac{1}{2}(i^2+j^2)-1+\frac{3}{2}(j-i), & \text{if $\diamondsuit=(i,j) \in S$}.
  \end{cases}
\end{equation}
\end{definition}

\bigskip

Comparing with equation \eqref{gamma} and Lemma \ref{discrep}, if $\diamondsuit=i, 1\leq i\leq k$, the modified discrepancies are same as the actual discrepancies; if  $\diamondsuit=(i,j) \in S$,
the modified discrepancies are different from the actual discrepancies. Let $d_{ij}$ be the actual discrepancy of $\widetilde{E_{ij}}$, where $1\leq i < j \leq k$. From Lemma \ref{discrep}, we know that $d_{ij}=\frac{1}{2}(i^2+j^2)-1$. Then we have $$a_{ij}=d_{ij}+\frac{3}{2}(j-i).$$

\bigskip
The following definition is motivated by the Euler characteristics of log-terminal pairs, considered, for example, by Batyrev in \cite{Bat3}

\begin{definition}\label{def.cst}
Clifford-stringy $E$-function of the variety $\widehat{Z}$ is defined by
$$
E_{cst}(Z;u,v):=\sum_{J\subseteq S} E(\widetilde{E}_J^{\circ};u,v) \prod_{\diamondsuit \in J} \frac {uv-1}{(uv)^{a_{\diamondsuit}+1}-1}.
$$
It is a priori only a rational function in fractional powers of $u,v$.
\end{definition}

Since $\chi(\widetilde{E}_J^{\circ})$ is the limit of $\sum_{J\subseteq S} E(\widetilde{E}_J^{\circ};u,v)$ as both $u$ and $v$ go to $1$, we get the definition of Clifford-stringy Euler Characteristics as follows.

\begin{definition}\label{defcli-stri}
Clifford-stringy Euler Characteristics of the variety $Z$ is defined by
$$\chi_{cst}(Z)=\sum_{J\subseteq S}\chi(\widetilde{E}_J^{\circ})\prod_{\diamondsuit \in J}\frac{1}{1+a_{\diamondsuit}}$$
where  $\chi$ denotes Euler characteristics with compact support.
\end{definition}

\medskip

\begin{remark}
The log resolution of singularities $\theta\circ\mu: \widetilde{Z} \rightarrow Z$ has the property that each $\widetilde{E}_J^{\circ}$ is a Zariski locally trivial fibration over its image $\omega=\theta\circ\mu(\widetilde{E}_J^{\circ})$ in $Z$. Thus, $\omega=\theta\circ\mu: \widetilde{Z} \rightarrow Z$ is a Zariski locally trivial log resolution. Thus, we can consider the local contributions to the Clifford-stringy Euler function and characteristics.
\end{remark}

\begin{definition}\label{contri}
For a point $y\in Z$, define the local contributions of $y$ to $E_{cst}(Z;u,v)$ and $\chi_{cst}(Z)$ to
be (respectively)
$$
S_{cst}(y;u,v):=\sum_{J\subseteq S} E(\widetilde{E}_J^{\circ}\cap \omega^{-1}(y);u,v) \prod_{\diamondsuit\in J} \frac {uv-1}{(uv)^{a_{\diamondsuit}+1}-1},
$$
$$
S_{cst}(y)=\sum_{J\subseteq S}\chi(\widetilde{E}_J^{\circ}\cap \omega^{-1}(y))\prod_{\diamondsuit \in J}\frac{1}{1+a_{\diamondsuit}}.$$
\end{definition}

\begin{remark}\label{cstofZ}
Thus defined $S(y;u,v)$ is a constructible function on $Z$ with values in the field of rational functions in $u$ and $v$. Indeed, if $y_1$ and $y_2$ are such that the set of $J$ with $\mu(\widetilde{E}_J^{\circ})$ that contain $y_1$ is the same as those that contain $y_2$, then $S(y_1;u,v)=S(y_2;u,v)$.
Similar as the equation \eqref{Zarlogtriv} for stringy $E$-function, there holds
\begin{equation}\label{cli-Zarlogtriv}
E_{cst}(Z;u,v) = \sum_{i} E(Z_i;u,v) S_{cst}(y\in Z_i;u,v),
\end{equation}
where $Z=\bigsqcup_i Z_i$ is the stratification of $Z$ into the sets on which $S$ is constant.
Also, we have $$\chi_{cst}(Z) = \sum_{i} \chi(Z_i) S_{cst}(y\in Z_i).$$
\end{remark}

Our main result is Theorem \ref{main} which states that 
$$\chi_{cst}(Z)=\chi(Y_W)$$
in any dimension. We will work our way towards it by computing the local contributions of points on $Z$ to the Clifford-stringy Euler characteristics.

\section{Description of fibers of $\theta$.}\label{fiber}
In this subsection, we give a description of the fibers of $\theta:\widehat{Z}\to Z$ and the strata of the stratification induced by $E_i$ and calculate their Euler characteristics.

\smallskip
We start with a preliminary proposition that finds the Euler characteristics of spaces of non-degenerate quadrics.

\begin{proposition}\label{chinondeg}
Let $M$ be a complex vector space of dimension $s$ and let $\Pp\Sym^2 M^*$ be the space of quadrics in $\mathbb{P}M$. There is an open subset 
$$(\Pp\Sym^2 M^*)_{nondeg}\subseteq \Pp\Sym^2 M^*$$
which consists of nondegenerate quadrics. Then its Euler characteristics (with compact support) is given by
\begin{equation}\label{nondege}
\chi((\Pp\Sym^2 M^*)_{nondeg})=\begin{cases}
    1, & \text{ when } s=1,2 ;\\
    0, & \text{ when } s\geq 3.
  \end{cases}
\end{equation}

\end{proposition}

\begin{proof}
Since every quadric in $\Pp M$ can be identified with a nondegenerate quadric in $\Pp (M/M_1)$ for some subspace $0\subseteq M_1\subsetneq M$, we get a relation
$$
\chi(\Pp\Sym^2 M^*) = \sum_{0\leq i<s} \chi(Gr(i,M)) \,\chi( \Pp\Sym^2 \Cc^{s-i})_{nondeg}
$$
which leads to
\begin{equation}\label{rec}
\frac {s(s+1)}2 =  \sum_{0\leq i<s} 
\begin{pmatrix} 
s \\
i
\end{pmatrix}
\,\chi( \Pp\Sym^2 \Cc^{s-i})_{nondeg}.
\end{equation}
The equation \eqref{rec} can be viewed as a recursive relation on $r_s = \chi( \Pp\Sym^2 \Cc^{s})_{nondeg}$.
For $s=1$ it implies $r_1=1$. Thus it suffices to show that $r_1=r_2=1,~r_{\geq 3}=0$ satisfies \eqref{rec},
which is simply the identity
$$
\frac {s(s+1)}2 = \begin{pmatrix} 
s \\
s-1
\end{pmatrix} + 
\begin{pmatrix} 
s \\
s-2
\end{pmatrix}.
$$
\end{proof}
\smallskip

We define the strata of $\widehat{\Pp W}$ and $\widehat {Z}$ as usual. Namely, for a set
$I\subseteq\{1,2,\dots, k\}$, we define
$$E_I=\cap_{{i}\in I}E_{i}\text{,  } E_I^\circ=\cap_{{i}\in I}E_{i}-\cup_{{i}\in \{1,2,\dots,k\}-J}E_{{i}},$$
$$T_I=\cap_{{i}\in I}T_{i}\text{,  }T_I^\circ=\cap_{{i}\in I}T_{i}-\cup_{{i}\in \{1,2,\dots,k\}-J}T_{{i}}.$$

\smallskip

The next definition is important in selecting the strata that have nonzero Euler characteristics.
\begin{definition}\label{star}
We say that  a subset  $I=\{i_1,i_2,\dots,i_s\}\subseteq\{1,2,\dots, k\}$, where $i_1<i_2<\dots<i_s$,
satisfies condition $(*)$
if
$i_j-i_{j-1}\le 2$  for  $j=1,\dots,s$  (with $i_0=0$ by convention).
\end{definition}

\begin{proposition}
For some $1\leq t \leq k$, let $y$ be a point in $ \phi_{t,Z}^{\circ}$, i.e. $\sigma(y)\in \Pp W$ corresponds to a quadric of corank $t$. Let $I=\{i_1,i_2,\dots,i_s\}$, where $i_1<i_2<\dots<i_s=t$ be a subset of $\{1,\ldots, k\}$. We consider the stratum $T_I^\circ\subseteq Z$ that corresponds to $I$. If $I$ satisfies condition $(*)$ of Definition \ref{star}, we have
\begin{equation}\label{chiTI}
\chi(T_I^\circ\cap  \pi^{-1}(\sigma(y)))
 =t!\Big(\frac{1}{2}\Big)^{t-|I|}
\end{equation}
and we have $\chi(T_I^\circ\cap  \pi^{-1}(\sigma(y)))=0$ if $I$ does not satisfy $(*)$. Here we use $\chi$ to denote the Euler characteristics with compact support.
\end{proposition}
\begin{proof}
For $y\in \phi_{t,Z}^{\circ}$, $\sigma(y)$ corresponds to a quadratic form $\mathbb C q\in \mathbb PW$ with corank $t$.
An element in $T_I^\circ\cap \pi^{-1}(\sigma(y))$ is a complete quadric which is mapped to $\mathbb C q$ under the map $\pi:\widehat{\mathbb PW}\to \mathbb PW$.
Then we have 
$$T_I^\circ\cap \pi^{-1}(\sigma(y))=\{0=F^0\subset F^1\subset \dots \subset F^s\subset F^{s+1}=V,  \mathbb{C}q_j\in \mathbb P\Sym^2(F^j/F^{j-1})^*, 
$$
$$
j=1,2,\dots,s|\dim F^j=i_j,\mathbb Cq\in \mathbb P\Sym^2(F^{s+1}/F^{s})^*\}$$
where all $\mathbb{C}q_j$ are nondegenerate and $F^s$ is the radical of $q$.
Thus we obtain the formula for the Euler characteristics of $T_I^\circ\cap \pi^{-1}(\sigma(y))$ 
$$
\chi(T_I^\circ\cap \pi^{-1}(\sigma(y)))=\chi(Fl(t; i_1,i_2-i_1,\dots,i_s-i_{s-1}))\prod_{j=1}^s \chi((\mathbb P\Sym^2 \Cc^{i_j-i_{j-1}})_{nondeg}),
$$ 
in terms of the Euler characteristics of the partial flag variety and the Euler characteristics of the spaces of smooth quadrics considered in Proposition \ref{chinondeg}.

\smallskip
By Proposition \ref{chinondeg}, $\chi((\mathbb P\Sym^2 M^*)_{nondeg})$ is zero when $\dim M\ge 3$ and is $1$ when $\dim M=1,2$. Therefore, $\chi(T_I^\circ\cap \pi^{-1}(\sigma(y)))=0$ if $I$ does not satisfy condition $(*)$ and is equal to $\chi(Fl(t; i_1,i_2-i_1,\dots,i_s-i_{s-1}))$ if $I$ satisfies it.
It remains to observe that in the latter case
$$
\chi(Fl(t; i_1,i_2-i_1,\dots,i_s-i_{s-1}))=\frac {t!} {\prod_{j}(i_j-i_{j-1})}
 =t!\Big(\frac{1}{2}\Big)^{t-|I|}.
$$
\end{proof}

\medskip
For a subset $I$ as above let us denote 
$I_\odd =\{i\in I| i\text{ is odd}\}$.
Then we have the following corollary.

\begin{corollary}\label{cchiEI}
Let $y$ be a point in $ \phi_{t,Z}^{\circ}$ for some $1\leq t \leq k$, and let $I=\{i_1,i_2,\dots,i_s\}$, where $i_1<i_2<\dots<i_s=t$. Recall that $E_I^\circ$ is the stratum of $\widehat{Z}$ that corresponds to $I$. If $I$ satifies $(*)$, we have
\begin{equation}\label{chiEI}
  \chi(E_I^\circ\cap \theta^{-1}(y))=\begin{cases}
    t!(\frac{1}{2})^{t-|I|},, & \text{ if } I_\odd \neq\emptyset;\\
    2t!(\frac{1}{2})^{t-|I|}, & \text{ if } I_\odd =\emptyset.
  \end{cases}
\end{equation}
and we have $\chi(E_I^\circ\cap \theta^{-1}(y))=0$ otherwise.
\end{corollary}

\begin{proof}
Since $\rho$ is a double cover which is ramified over $\{T_i| $ i is  odd$\}$, we have
\[\begin{split}
&\chi(E_I^\circ\cap \theta^{-1}(y))=\chi(T_I^\circ\cap  \pi^{-1}\circ\sigma(y)) \text{ if } I_\odd \neq\emptyset,\\
&\chi(E_I^\circ\cap \theta^{-1}(y))=2\chi(T_I^\circ\cap  \pi^{-1}\circ\sigma(y)) \text{ if } I_\odd =\emptyset.
\end{split}\]
Then by \eqref{chiTI}, we get \eqref{chiEI}.
\end{proof}

\section{The local contributions to the Clifford-stringy Euler characteristics of $Z$.}\label{contribu}
In this section, we calculate the local contribution of points on the strata $\phi^{\circ}_{t,Z}$ to the Clifford-stringy Euler characteristics of $Z$, where $\phi^{\circ}_{t,Z}=\phi_{t,Z}-\phi_{(t-1),Z}$ and $t \in \{1,2,\ldots,k\}$. We denote the local contribution by $\chi_t$. Then we will deduce
a better expression of the formula of the local contribution in Proposition \ref{chit}. Finally, we calculate it for small $t$ to observe an interesting $1,2,1,2, \ldots$ pattern, which will be proved in the next section.

\medskip

First, we recall the notations.
In Section \ref{Doub} and \ref{resoZ},
we have obtained a resolution of singularities of $Z$: $\theta\circ\mu: \widetilde{Z}\to \widehat{Z}\to Z$. Let us consider the set $S=\{i| 1\le i\le k\}\cup \{(i,j)|1\le i<j\le k, i,j\text{ is odd }\}$. All the exceptional divisors of $\widetilde{Z}$ are $\{\widetilde{E}_{\diamondsuit}| \diamondsuit\in S\}$.
For a set $J\subseteq S$, we define $$\mu(S)=\{\text{All the integers appear in } J\}$$ which is a subset of $\{1,2,\dots,k\}$.
For a set $I\subseteq\{1,2,\dots, k\}$, we define $\theta(I)=\max_{i\in I} i$.
For $J\subseteq S$ and $I\subseteq \{1,2,\dots, k\}$, we define
\[\begin{split}
&\widetilde{E}_J=\cap_{\diamondsuit\in J}\widetilde{E}_\diamondsuit \text{,  } \widetilde{E}_J^\circ=\cap_{\diamondsuit\in J}\widetilde{E}_\diamondsuit-\cup_{\diamondsuit\in S-J}\widetilde{E}_{\diamondsuit},\\
&E_I=\cap_{\diamondsuit\in I}E_\diamondsuit \text{,  } E_I^\circ=\cap_{\diamondsuit\in I}E_\diamondsuit-\cup_{\diamondsuit\in \{1,2,\dots,k\}-J}E_{\diamondsuit},\\
&T_I=\cap_{\diamondsuit\in I}T_\diamondsuit \text{,  } T_I^\circ=\cap_{\diamondsuit\in I}T_\diamondsuit-\cup_{\diamondsuit\in \{1,2,\dots,k\}-J}T_{\diamondsuit}.
\end{split}\]
For sets $J$ of integers also use the notation:
\begin{itemize}
\item $J_{\even}=\{i\in J| i\text{ is even}\}$;
\item $J_{\odd}=\{i\in J| i\text{ is odd}\}$.
\end{itemize}

\medskip
Given an arbitrary element $t \in \{1,2,\ldots,k\}$, we will define a function $G$ on the subsets of $\{1 \leq i \leq t| i $ is odd$\}$.
\begin{definition}\label{Grec}
Let $T \subseteq \{1 \leq i \leq t| i $ is odd$\}$. We define $G$ recursively by
\[\begin{split}
&G[T]=\frac{1}{r^2}, \text { when } T=\{r\} \text { and }\\
&G[T]=\frac{1}{1+a_{\min(T),\max(T)}}(G[T-\max(T)]+G[T-\min(T)]),
\end{split}\]
 where $a_{\min(T),\max(T)}$ is defined in Definition \ref{aij}.
We also 
define $F[T]=G[T]2^{|T|-1}$ which satisfies
\[\begin{split}
&F[T]=\frac{1}{r^2}, \text{ when } T=\{r\}\text { and }\\
&F[T]=\frac{2}{1+a_{\min(T),\max(T)}}(F[T-\max(T)]+F[T-\min(T)]).
\end{split}\]
\end{definition}

\begin{example} We illustrate Definition \ref{Grec} with two examples. For simplicity we write 
$G[1,3]$ instead of $G[\{1,3\}]$ and similarly for other sets and for $F$.
Recall that $a_{i,j}=\frac 12(i^2+j^2)-1+\frac 32(j-i)$, so $a_{1,3}=7$, $a_{3,5}=19$, $a_{5,7}=39$ and $a_{3,7}= 35$. We have 
\[\begin{split}
G[1,3]&=\frac{1}{1+a_{1,3}}(G[1]+G[3])= \frac{1}{8}(\frac 1{1} + \frac 1{9}) = \frac {5}{36};
\\
F[1,3]&=\frac 5{18};
\\
G[3,5,7]&=\frac{1}{1+a_{3,7}}(G[3,5]+G[5,7]) = \frac 1{36}(G[3,5]+G[5,7]) 
\\
&=\frac 1{36}\left(\frac 1{(1+a_{3,5})}(\frac 19 + \frac 1{25}) + \frac 1{(1+a_{5,7})}(\frac 1{25} + \frac 1{49})\right)
\\
&=\frac 1{36}\left(\frac 1{20}(\frac 19 + \frac 1{25}) + \frac 1{40}(\frac 1{25} + \frac 1{49})\right) = \frac {1999}{7938000};
\\
F[3,5,7]&=\frac{1999}{1984500}.
\end{split}\]
\end{example}

The next proposition explains why we choose to consider the fairly obscure recursive relations of Definition \ref{Grec}.
\begin{proposition}\label{chit}
The local contribution of $y \in \phi_{t,Z}^\circ$ to the Clifford-stringy Euler characteristics of $Z$ is given by
\begin{equation}\label{chitodd}
\chi_t=\sum_{T\subseteq\{1,3,5,\dots,t\},t\in T}t! F[T] \prod_{\substack{1\le r\le t,2|r,\\ r-1 \notin T \text{ or }r+1 \notin T}}\frac{1}{r^2}
\prod_{\substack{1\le r\le t,2|r, \\r-1 \in T,r+1\in T}}(\frac{1}{2^2}+\frac{1}{r^2})
\end{equation}for odd $t$ and by
\begin{equation}\label{chiteven}
\chi_t={2\, t!}\prod_{1\le r\le t,2|r }\frac{1}{r^2}+\sum_{\emptyset\ne T\subseteq\{1,3,5,\dots,t-1\}}
\frac { 2\,t!\,F[T]}{t^2} \prod_{\substack{1\le r<t,2|r,\\ r-1 \notin T \text{ or }r+1 \notin T}}\frac{1}{r^2}
\prod_{\substack{1\le r< t,2|r,\\ r-1 \in T,r+1\in T}}(\frac{1}{2^2}+\frac{1}{r^2})
\end{equation}for even $t$.
\end{proposition}

\begin{proof}
Let us first consider the case of odd $t$.
The fibers over $y$ of strata on $\widetilde Z$ map to strata on $\widehat Z$. We can collect these contributions according to the corresponding subset $I$ of $\{1,\ldots, t\}$. We observe that 
the contribution of the strata that map to  $E_I^\circ$ is  equal to zero if $I$ does not satisfy the condition $(*)$ 
of Definition \ref{star}  that the minimum element is at most $2$ and difference between consecutive elements is at most $2$. If $I$ satisfies $(*)$ then we claim that the contribution is equal to
\begin{equation}\label{toprove1}
 t! F[T] \Big(\prod_{r\in I_\even} \frac 1{r^2}\Big)  \Big(\prod_{2|r,\,r\not\in I_\even } \frac 1{2^2}\Big).
\end{equation}
Here $T=I_\odd $ is the set of all odd elements of $I$, which has to contain $t$.
This would prove the equation \eqref{chitodd}, because for a given $T=I_\odd $ the condition $(*)$  means that $I_\even$ must include all even $r$ such that either of $r+1$ or $r-1$ is not in $T$, but it may or may not include even $r$ such that both $r+1$ and $r-1$ are in $T$.

 \smallskip
The fiber of $\widetilde Z\to \widehat Z$ over a point $\hat y\in E_I^\circ$ 
is isomorphic to a toric variety given by the fan $\Sigma$ constructed in Appendix \ref{res} to resolve the singularities of the intersection of divisors $E_i$ for $i\in I_\odd $. The stratification of $\widetilde Z$ induces precisely the toric stratification on the fiber over $\hat y$. Most of the strata are tori of dimension at least $1$ and thus do not contribute to the Clifford-stringy Euler characteristics $S_{cst}$ in Definition \ref{contri}. 
The contributions of zero-dimensional strata come with the coefficients
$$
\Big(\prod_{r\in I_\even}\frac 1{1+a_{\{r\}}} \Big) G[I_\odd ] = \Big(\prod_{r\in I_\even}\frac 2{r^2} \Big) 2^{1-|I_\odd |}F[I_\odd ] .
$$
Indeed, the relevant contribution $G_\varphi$ of  Proposition  \ref{GrecPhi} satisfies the recursions of 
$G[T]$ after coordinate change $s\to \frac {s+1}2$. 
 
\smallskip
To get the contribution for each $I$ we need to further multiply by the Euler characteristics of the corresponding stratum on $\widehat Z$, given in Corollary \ref{cchiEI} to get 
$$
 \Big(\prod_{r\in I_\even}\frac 2{r^2} \Big) 2^{1-|I_\odd |}F[I_\odd ] 
t!  (\frac{1}{2})^{t-|I|}= t! F[T] \Big(\prod_{r\in I_\even} \frac 1{r^2}\Big) 2^{|I_\even | +1-|I_\odd | -t+|I|}
$$
$$
= t! F[T] \Big(\prod_{r\in I_\even} \frac 1{r^2}\Big)  \Big(\prod_{2|r,\,r\not\in I_\even} \frac 1{2^2}\Big).
$$
Here we used that $|\{2|r,\,r\not\in I_\even \}|= \frac {(t-1)}2 - |I_\even | $. This proves \eqref{toprove1} and thus \eqref{chitodd}.

\smallskip
When $t$ is even, the analysis is very similar. One difference is that we now have $t\in I_\even $. Another
is that we need to separately consider the case of Corollary \ref{cchiEI} with empty $I_\odd $. 
Specifically, if $I_\odd \neq \emptyset$, the contribution of $I$ that satisfies $(*)$ is equal to
$$
 2\,t! F[T] \Big(\prod_{r\in I_\even } \frac 1{r^2}\Big)  \Big(\prod_{2|r,\,r\not\in I_\even } \frac 1{2^2}\Big)
$$
because $|\{2|r,\,r\not\in I_\even \}|= \frac t2 - |I_\even | $.
When $I_\odd =\emptyset$, we must have $I$ to be all even numbers in $\{1,\ldots,t\}$ and Proposition \ref{cchiEI} gives an extra
factor of $2$ to give
$$
2 t! 2^{-\frac 12 t} \prod_{1<r\leq t, 2|r} \frac 2{r^2} =
2 t!  \prod_{1<r\leq t, 2|r} \frac 1{r^2}.
$$
This proves \eqref{chiteven}.
\end{proof}

Proposition \ref{chit} allows us to calculate $\chi_t$ for small $t$.
\[\begin{split}
&\chi_1=F[1]=1\\
&\chi_2 =2\cdot 2! \frac{1}{2^2}+\frac {2\cdot 2!\, F[1]}{2^2}=2\\
&\chi_3 =3!\big(\frac{1}{2^2}F[3]+(\frac 1{2^2} + \frac 1{2^2})F[1,3]\big)=1\\
&\chi_4 =2\cdot 4!(\frac{1}{2^2})(\frac{1}{4^2})+2\cdot 4!\big((\frac{1}{2^2})(\frac{1}{4^2})F[1]+ (\frac{1}{2^2})(\frac{1}{4^2})F[3]+ (\frac{1}{2^2} + \frac{1}{2^2})(\frac{1}{4^2})F[1,3]\big)=2\\
&\chi_5 = 5! \big((\frac{1}{2^2}) (\frac{1}{4^2}) F[5]+ (\frac{1}{2^2}) (\frac{1}{2^2} + \frac{1}{4^2}) F[3, 5]+  (\frac{1}{2^2}) (\frac{1}{4^2}) F[1, 5]\\
&\hskip 20pt
+  (\frac{1}{2^2} + \frac{1}{2^2}) (\frac{1}{2^2}
 + \frac{1}{4^2}) F[1, 3, 5]\big)=  1
\end{split}
\]
It is reasonable to expect that this pattern continues and, in fact, we prove it in the next section by a delicate recursive argument.

\begin{remark}
The contributions $\chi_t$ to Clifford-stringy Euler characteristics are different from those for the usual Batyrev's stringy Euler characteristics for $t\geq 3$, where the first non-trivial $F[T]$ appears.
It is crucial to change the discrepancies of $\widetilde E_{i,j}$ by extra $\frac 32(j-i)$ to get this pattern. In fact, we discovered this change by looking for discrepancies that may make such pattern possible. 
\end{remark} 



\section{Proof of the formula for $\chi_t$.}\label{sec.hard}
The goal of this section is to prove the $1,2,1,2,...$ pattern observed in the previous section.
\begin{proposition}
\begin{equation}\label{gammaa}
  \chi_t=\begin{cases}
    1, & \text{if $1\leq t\leq k,t\text{ is odd}$};\\
    2, & \text{if $1\leq t\leq k,t\text{ is even}$}.
  \end{cases}
\end{equation}
\end{proposition}

Before starting the proof of Proposition \ref{gammaa}, we will prove several lemmas.
\begin{lemma}\label{evenodd}
For all $j\geq 1$ there holds
$$\chi_{2j}=\frac{1}{j}\chi_{2j-1}+\frac{2j-1}{2j}\chi_{2j-2}.$$
\end{lemma}

\begin{proof}
For any even number $s$ and $T\subseteq\{1,3,5,\dots,2j-1\}$, let $$H[T,s]=\left(\prod_{\substack{1\le r\leq s,2|r, \\r-1 \notin T \text{ or }r+1 \notin T}}\frac{1}{r^2}\right)
\left(\prod_{\substack{1\le r\leq s,2|r, \\r-1 \in T,r+1\in T}}(\frac{1}{2^2}+\frac{1}{r^2})\right).$$ Then according to Proposition \ref{chit}, we have
\begin{equation}\label{2j-1}
\begin{split}
&\chi_{2j-1}=(2j-1)!\sum_{\substack{T\subseteq\{1,3,5,\dots,2j-1\},\\2j-1\in T}}H[T,2j-2]F[T],\\
&\chi_{2j-2}=2 (2j-2)!\prod_{\substack{1\le r\le 2j-2,\\2|r }}\frac{1}{r^2}+\frac{2 (2j-2)!}{(2j-2)^2}\sum_{\emptyset\ne T\subseteq\{1,3,5,\dots,2j-3\}}H[T,2j-4]F[T].
\end{split}
\end{equation}
We separate the sum in 
$$\chi_{2j} = 2 (2j)!\prod_{1\le r\le 2j,2|r }\frac{1}{r^2}+\frac{2 (2j)!}{(2j)^2}\sum_{\emptyset\ne T\subseteq\{1,3,5,\dots,2j-1\}}H[T,2j-2]F[T]$$
into two parts according to whether $2j-1 \notin T$ or $2j-1 \in T$ to get
\[\begin{split}
& \chi_{2j}=2(2j)!\prod_{\substack{1\le r\le 2j,\\2|r }}\frac{1}{r^2}+\frac{2 (2j)! }{(2j)^2}\sum_{\substack{\emptyset\ne T\subseteq\{1,3,5,\dots,2j-1\}\\2j-1 \notin T}} H[T,2j-2]F[T]\\
& +\frac{2 (2j)!}{(2j)^2}\sum_{\substack{\emptyset\ne T\subseteq\{1,3,5,\dots,2j-1\}\\2j-1 \in T}} H[T,2j-2]F[T]\\
& =2(2j)!\frac{1}{(2j)^2}\prod_{\substack{1\le r\le 2j-2,\\2|r }}\frac{1}{r^2}+\frac{2(2j)!}{(2j)^2}\frac{1}{(2j-2)^2}\sum_{\substack{\emptyset\ne T\subseteq\{1,3,5,\dots,2j-3\}}} H[T,2j-4]F[T]\\
& +\frac{(2j-1)!}{j}\sum_{\substack{\emptyset\ne T\subseteq\{1,3,5,\dots,2j-1\}\\2j-1 \in T}} H[T,2j-2]F[T] =\frac{2j-1}{2j}\chi_{2j-2}+\frac{1}{j}\chi_{2j-1},
\end{split}\]where we use \eqref{2j-1} in the last step.
\end{proof}

\smallskip
It is more complicated to get a recursion of $\chi_t$ for odd $t$. 
It turns out that it is easier to get a recursive relation on a more general function $\delta[j,a]$ which we now introduce.

\smallskip
Recall that 
In Section \ref{contribu}, we defined the function $F$ on the sets of odd positive integers by 
\[\begin{split}
& F[x] := 1/x^2;\\
& F[x_1, x_2] := (F[x_1] + F[x_2])  \left(\frac 4{x_1^2 + x_2^2 + 3 (x_2 - x_1)}\right);\\
& F[x_1, x_2, x_3] := (F[x_1, x_2] +
     F[x_2, x_3]) \left(\frac 4{x_1^2 + x_3^2 + 3 (x_3 - x_1)}\right);\\
& \ldots \\
& F[x_1, \ldots, x_n] := (F[x_1,\ldots ,x_{n-1}] +
     F[x_2, \ldots , x_n])  \left(\frac 4{x_1^2 + x_n^2 + 3 (x_n - x_1)}\right).
\end{split}\]
\begin{remark}\label{F[T]}
We can clearly extend this definition to arbitrary $x_i$, provided that we do not divide by zero.
It is also clear from the definition that $F[x_1, \ldots, x_n]=F[-x_n , \ldots, -x_1]$.
\end{remark}

\begin{definition}
Let $T=\{x_1,\dots,x_n\}$ be a set of integers and $x_1<x_2<\dots<x_n$. Define $$F[T+a]:=F[x_1+a,x_2+a,\dots,x_n+a]$$as a rational function of $a$.
For arbitrary two even numbers $s_1 \leq s_2$, let
\begin{equation}\label{HT}
H[T,s_1,s_2,a]=\prod_{\substack{s_1\leq r\leq s_2,2|r,\\ r+1,r-1\in T}}(\frac{1}{4}+\frac{1}{(r+a)^2})\prod_{\substack{s_1\leq r\leq s_2,2|r,\\ r+1 \notin T \text{ or } r-1 \notin T}}\frac{1}{(r+a)^2}.
\end{equation}
We define for $j\geq 1$
$$
\delta[j,a]=
\sum_{\substack{T\subseteq\{1,3,5,\dots,2j+1\},\\ 2j+1\in T}}F[T+a]H[T,2,2j,a].
$$
\end{definition}

\begin{lemma}\label{jas}
We have
\begin{equation}
\sum_{\substack{T\subseteq\{1,3,\dots,2j+1\},\\1\in T,2j+1\in T}}F[T+a]H[T,2,2j,a]=\delta[j,a]-\frac{1}{(a+2)^2}\delta[j-1,a+2].
\end{equation}
\end{lemma}
\begin{proof}
We separate $\delta[j,a]$ into two parts according to whether $1 \in T$ or $1 \notin T$. If $1 \notin T$, the corresponding term for $\delta[j,a]$ differs from the term for $T-2$ for $\delta[j-1,a+2]$ by the factor $\frac{1}{(a+2)^2}$ that comes from the second product in \eqref{HT} for $r=2$.
\end{proof}


\begin{lemma}\label{j-a}
We have
$$\delta[j,-a-2-2j]=\sum_{\substack{T\subseteq\{1,3,\dots,2j+1\},\\1\in T}}F[T+a]H[T,2,2j,a].$$
\end{lemma}
\begin{proof}
By definition of $\delta$, we have
$$
\delta[j,-a-2-2j]=\sum_{\substack{T\subseteq\{1,3,\dots,2j+1\},\\2j+1\in T}}F[T-a-2-2j]H[T,2,2j,-a-2-2j].
$$
The map $T\mapsto (2j+2)-T$ provides a bijection between the set of subsets $T$ of $\{1,3,\dots,2j+1\}$ that contain $(2j+1)$ and
the set of subsets $T$ that contain $1$. Therefore,
it suffices to show that for any $T\subseteq \{1,3,\dots,2j+1\}$ there hold
$
F[T-a-2-2j] = F[(2j+2-T)+a]$ and $H[T,2,2j,-a-2-2j]=H[(2j+2-T),2,2j,a]
$.
The first statement is the consequence of Remark  \ref{F[T]}, and the second statement follows from
$$
\frac 1{(r-a-2-2j)^2} = \frac1{((2j+2-r)+a)^2}.
$$
\end{proof}


The next proposition is a key technical step of the paper.
\begin{proposition}\label{recursive}
For any $j\geq 2$, we have the recursive relation
\begin{equation}\label{recurformu}
\begin{split}
\delta[j+1,a]&=\big((1+\frac{4}{c})\frac{1}{(a+2)^2}+\frac{1}{c}\big)\delta[j,a+2]-\frac{1}{c}\frac{1}{(a+4)^2}\delta[j-1,a+4])\\
&+\frac{1}{c}\delta[j,a]-\frac{1}{c}\frac{1}{(a+2)^2}\delta[j-1,a+2]+\frac{4}{c}\frac{1}{(a+2j+2)^2}\delta[j,-a-2j-2],
\end{split}
\end{equation}
where $c=(a+1)^2+(a+2j+3)^2+6j+6$.
\end{proposition}
\begin{proof}

First, we observe that
$$
\delta[j+1,a]=\sum_{\substack{T\subseteq\{1,3,5,\dots,2j+3\},\\ 2j+3\in T}}(F[T+a]H[T,2,2j+2,a])
$$ 
can be separated into two parts according to whether $1\in T$ or $1 \notin T$ as

$$
\sum_{\substack{T\subseteq\{3,5,\dots,2j+3\},\\ 2j+3\in T}}(F[T+a]H[T,2,2j+2,a])+\sum_{\substack{T\subseteq\{1,3,5,\dots,2j+3\},\\ 1,2j+3\in T}}(F[T+a]H[T,2,2j+2,a]).
$$
By shifting, the first part equals
\[\begin{split}
&\sum_{\substack{T\subseteq\{1,3,\dots,2j+1\},\\ 2j+1\in T}}(F[T+a+2]H[T,0,2j,a+2])\\
&=\frac{1}{(a+2)^2}\sum_{\substack{T\subseteq\{1,3,\dots,2j+1\},\\ 2j+1\in T}}(F[T+2+a]H[T,2,2j,2+a])=\frac{1}{(a+2)^2}\delta[j,a+2].
\end{split}\]
Then we write the second part by the recursive relation of $F$ to get
\begin{equation}\label{a}
\sum_{\substack{T\subseteq\{1,3,5,\dots,2j+3\},\\1,2j+3\in T}}\big((F[\{T-\{2j+3\}\}+a]+F[\{T-\{1\}\}+a])
H[T,2,2j+2,a]\big)(\frac{4}{c}),
\end{equation}
where $c=(a+1)^2+(a+2j+3)^2+6j+6$.
We separate the second part of \eqref{a}
\[(\frac{4}{c})\sum_{\substack{T\subseteq\{1,3,5,\dots,2j+3\},\\1,2j+3\in T}}F[\{T-\{1\}\}+a]H[T,2,2j+2,a]\]
into two parts according to whether $3\in T$ or $3 \notin T$ as
\begin{equation*}
\begin{split}
&(\frac{4}{c})\sum_{\substack{T\subseteq\{1,3,5,\dots,2j+3\},\\1,2j+3\in T,3\in T}}F[\{T-\{1\}\}+a]H[T,2,2j+2,a]\\
&+(\frac{4}{c})\sum_{\substack{T\subseteq\{1,3,5,\dots,2j+3\},\\1,2j+3\in T, 3\notin T}}F[\{T-\{1\}\}+a]H[T,2,2j+2,a]\\
&=(\frac{4}{c})(\frac{1}{4}+\frac{1}{(a+2)^2})\sum_{\substack{T\subseteq\{3,5,\dots,2j+3\},\\2j+3\in T,3\in T}}F[T+a]H[T,4,2j+2,a]\\
&+(\frac{4}{c})(\frac{1}{(a+2)^2})\sum_{\substack{T\subseteq\{3,5,\dots,2j+3\},\\2j+3\in T, 3\notin T}}F[T+a]H[T,4,2j+2,a],
\end{split}
\end{equation*}
thus, the second part of formula \eqref{a} equals


\[\begin{split}
&(\frac{4}{c})(\frac{1}{(a+2)^2})\sum_{\substack{T\subseteq\{3,5,\dots,2j+3\},\\2j+3\in T}}F[T+a]H[T,4,2j+2,a]\\
&+(\frac{4}{c})\frac{1}{4}\sum_{\substack{T\subseteq\{3,5,\dots,2j+3\},\\2j+3\in T,3\in T}}F[T+a]H[T,4,2j+2,a]
\end{split}\]
\begin{equation}\label{u1}
=(\frac{4}{c})\big(\frac{1}{(a+2)^2}\delta[j,a+2]+\frac{1}{4}(\delta[j,a+2]-\frac{1}{(a+4)^2}\delta[j-1,a+4])\big).
\end{equation}
The last step is obtained by Lemma \ref{jas} and shifting.

\medskip

Let's consider the first part of formula \eqref{a}. We separate
\begin{equation*}
(\frac{4}{c})\sum_{\substack{T\subseteq\{1,3,5,\dots,2j+3\},\\1,2j+3\in T}}F[\{T-\{2j+3\}\}+a]H[T,2,2j+2,a]
\end{equation*}
into two parts according to whether $2j+1 \in T$ or $2j+1 \notin T$ and use the same process as above to get

$$(\frac{4}{c})(\frac{1}{(a+2j+2)^2})\sum_{\substack{T\subseteq\{1,3,\dots,2j+1\},\\1\in T}}F[T+a]H[T,2,2j,a]
+(\frac{4}{c})\frac{1}{4}\sum_{\substack{T\subseteq\{1,3,\dots,2j+1\},\\1\in T,2j+1\in T}}F[T+a]H[T,2,2j,a]$$

\begin{equation}\label{u2}
=(\frac{4}{c})\big(\frac{1}{(a+2j+2)^2}\delta[j,-a-2-2j]
+\frac{1}{4}(\delta[j,a]-\frac{1}{(a+2)^2}\delta[j-1,a+2])\big).
\end{equation}
The last step is obtained by Lemma \ref{jas} and Lemma \ref{j-a}.

\medskip

Combining formula \eqref{u1} and formula \eqref{u2}, we have
\begin{align*}
\delta[j+1,a]=&\frac{1}{(a+2)^2}\delta[j,a+2]+(\frac{4}{c})\Big(\frac{1}{(a+2)^2}\delta[j,a+2]
+\frac{1}{4}(\delta[j,a+2]
\\
&
-\frac{1}{(a+4)^2}\delta[j-1,a+4])
+\frac{1}{4}(\delta[j,a]-\frac{1}{(a+2)^2}\delta[j-1,a+2])
\\
&
+\frac{1}{(a+2j+2)^2}\delta[j,-a-2-2j]\Big),
\end{align*}
which implies the statement of the proposition.
\end{proof}
\bigskip
We can use Proposition \ref{recursive} to find an explicit formula for $\delta[j,a]$.
\begin{corollary}\label{deltarecu}
\begin{equation}\label{delta}
\delta[j,a]=\frac{a\sum_{i=0}^{j}\frac{1}{(a+2i)(a+1+2i)}}{\prod_{i=1}^{2j+1}(a+i)}
\end{equation}
\end{corollary}
\begin{proof}
We  use induction on $j$.

\smallskip
Base. The cases $j=1$ and $j=2$ are checked by calculating $\delta[j,a]$ directly.

\smallskip
Induction step. It suffices to check that the right hand side of equation \eqref{delta} satisfies the recursion of Proposition \ref{recursive}. Equivalently, it suffices to show that $\phi'[j,a]=\sum_{i=0}^{j}\frac{1}{(a+2i)(a+1+2i)}$ satisfies the recursion
\begin{equation}\label{phirecur}
\begin{split}
a\phi[j+1,a]=&\big((1+\frac{4}{c})\frac{1}{(a+2)^2}+\frac{1}{c}\big)(a+1)(a+2)^2\phi[j,a+2]
\\&
-\frac{1}{c}(a+1)(a+2)(a+3)\phi[j-1,a+4]
\\&
+\frac{1}{c}a(a+2j+2)(a+2j+3)\phi[j,a]
\\&
-\frac{1}{c}(a+1)(a+2j+2)(a+2j+3)\phi[j-1,a+2]
\\
&+\frac{4}{c}(a+2j+3)\phi[j,-a-2j-2]
\end{split}
\end{equation} obtained by passing from $\delta[j,a]$ to $\phi[j,a]=\frac{1}{a}\delta[j,a]\prod_{i=1}^{2j+1}(a+i)$.
\medskip

We observe that
\[\begin{split}
&\phi'[j-1,a+4]=\phi'[j,a+2]-\frac 1{(a+2)(a+3)},\\
& \phi'[j,a]=\phi'[j,a+2]+\frac{1}{a(a+1)}-\frac{1}{(a+2j+2)(a+2j+3)},\\
& \phi'[j-1,a+2]=\phi'[j,a+2]-\frac{1}{(a+2j+2)(a+2j+3)},\\
& \phi'[j,-a-2-2j]=\phi'[j,a+1]=-\phi'[j,a+2]+\frac{1}{a+1}-\frac{1}{a+2j+3},\\
& \phi'[j+1,a]=\phi'[j,a+2]+\frac{1}{a(a+1)}
\end{split}\]
directly from the definition of $\phi'[j,a]$.
Then we see that both sides of the equation \eqref{phirecur} for $\phi'$ equal $a\phi'[j,a+2]+\frac{1}{(a+1)}$, which proves the induction step.
\end{proof}
Now we are ready to prove Proposition \ref{gammaa}.
\begin{proof}[Proof of Proposition \ref{gammaa}]
By  Corollary \ref{deltarecu}, we see $\chi_{2j+1}=(2j+1)!\delta[j,0]=(2j+1)!\frac{1}{(2j+1)!}=1$ for all j.
Then by induction on $j$ we see that $\chi_{2j}=2$ in view of Lemma \ref{evenodd}.
\end{proof}

\medskip
\section{Relation with Euler characteristics.}\label{same}
In this section, we prove our main Theorem \ref{main}, which exhibits the relation between the Clifford-stringy Euler characteristics and the Euler characteristics of the complete intersection of quadrics.
As in Section \ref{Doub}, let $V$ be a vector space of dimension $2n$.
\begin{lemma}\label{Qt}
Let $Q_t$ be a quadric in $\mathbb{P}V$ of rank $2n-t$. Then we have
\begin{equation*}
\chi(Q_t)=\begin{cases}
     2n-1, & \text{if $t$  is odd};\\
   2n, & \text{if $t$ is even}.
  \end{cases}
\end{equation*}
\end{lemma}

\begin{proof}
Let $Q$ be a smooth quadric in $\Pp^r$. It is known that
\begin{equation}\label{chiQ}
  \chi(Q)=\begin{cases}
    \dim Q + 1, & \text{if $\dim Q$  is odd};\\
  \dim Q + 2, & \text{if $\dim Q$ is even}.
  \end{cases}
\end{equation}
If $t>0$, we know that $Q_t$ is a cone from the projective space $\mathbb{P}Rad(Q_t)$ to a smooth quadric $Q'_t$ of dimension $2n-2-t$, see \cite{Add}.
 So $Q_t\backslash \mathbb{P}Rad(Q_t)$ is a fiber bundle over the smooth quadric $Q'_t$ of dimension $2n-2-t$ with fibers $\mathbb{C}^t$. Thus by equation \eqref{chiQ}, we have
 \begin{equation*}
\chi(Q_t)=\chi(\mathbb{P}Rad(Q_t))+\chi({Q_t}')
  =t+(2n-2-t)+\begin{cases}
     1, & \text{if $2n-2-t$  is odd};\\
   2, & \text{if $2n-2-t$ is even},
  \end{cases}\\
 \end{equation*}
which implies the statement of the lemma.

\end{proof}
As before, let
$W$ be a generic subspace of $\Sym^2V^*$ of dimension $n$ and $\sigma: Z\rightarrow \mathbb{P}W$ be the double cover ramified over $\phi_1$. 
We have the following theorem.

\begin{theorem}\label{main}
Let $Y=\{v\in V| q(v)=0$ for all $q\in \mathbb{P}W\}$ be the complete intersection associated to $W$ in $\mathbb{P}V \cong \mathbb{P}^{2n-1}$. Then we have $$\chi_{cst}(Z)=\chi(Y).$$
\end{theorem}

\begin{proof}
Consider the universal quadric $\mathbb{H}=\{(q,v)\in \mathbb{P}W\times \mathbb{P}V|q(v)=0\}.$ We will calculate its Euler characteristics in two ways.

\smallskip
First, we consider the projection $\eta:\mathbb{H}\rightarrow \mathbb{P}V$ such that $\eta((q,v))=v$ for $(q,v)\in\mathbb{H}$.
When $v \in Y$, $\eta^{-1}(v)\cong\mathbb{P}W\cong\mathbb{P}^{n-1}$. When $v \in \mathbb{P}V-Y$, $\eta^{-1}(v)\cong\{\mathbb{C}q \in \mathbb{P}W | q(v)=0\}$ which is a hyperplane in $\mathbb{P}W$. So $\eta^{-1}(v)\cong \mathbb{P}^{n-2}$. Thus we have
\begin{equation}\label{Hpv}
\begin{split}
&\chi(\mathbb{H})=\chi(\mathbb{P}^{n-1})\chi(Y)+\chi(\mathbb{P}^{n-2})\chi(\mathbb{P}V-Y)\\
&=n\chi(Y)+(n-1)\chi(\mathbb{P}^{2n-1})-(n-1)\chi(Y)=2n(n-1)+\chi(Y).
\end{split}
\end{equation}

\smallskip
Second, we consider the projection $\eta':\mathbb{H}\rightarrow \mathbb{P}V$ such that $\eta'((q,v))=q$ for $(q,v)\in\mathbb{H}$.
In section \ref{Doub}, we have defined $\phi_t$, $t=1,2,\ldots,k$. Let $\phi_0=\mathbb{P}W$. Denote $\phi^{\circ}_t=\phi_t-\phi_{t+1}$, $0\leq t\leq k$
and $\phi^{\circ}_k=\{q\in\mathbb{P}W| \corank\,(q)=k \}$. Then we have $\mathbb{P}W=\bigsqcup_{t=0,1,\ldots,k}\phi^{\circ}_t$ and $\chi(\mathbb{P}W)=\sum_{t=0,1\ldots,k}\chi(\phi^{\circ}_t)$ When $q\in \phi^{\circ}_t$,
$\eta'(q)\cong \{v\in \mathbb{P}V | q(v)=0\}$ which we denote by $Q_t$. We know $Q_t$ is a quadric in $\mathbb{P}V$ with rank $2n-t$.
 Thus, by Lemma \ref{Qt}, we know that $\chi(Q_t)=2n$ when $t$ is even and $\chi(Q_t)=2n-1$ when $t$ is odd, and we have
\[
\begin{split}
\chi(\mathbb{H})&=\sum_{0\leq t\leq k}\chi(Q_t)\chi(\phi^{\circ}_t)\\
&=\sum_{2|t,0\leq t\leq k}2n\chi(\phi^{\circ}_t)+\sum_{2\nmid t,0\leq t\leq k}(2n-1)\chi(\phi^{\circ}_t)\\
&=(2n-2)\sum_{0\leq t\leq k}\chi(\phi^{\circ}_t)+\sum_{2|t,0\leq t\leq k}2\chi(\phi^{\circ}_t)+\sum_{2\nmid t,0\leq t\leq k}\chi(\phi^{\circ}_t)\\
&=(2n-2)\chi(\mathbb{P}W)+\sum_{2|t,0\leq t\leq k}2\chi(\phi^{\circ}_t)+\sum_{2\nmid t,0\leq t\leq k}\chi(\phi^{\circ}_t)\\
&=(2n-2)n+\sum_{2|t,0\leq t\leq k}2\chi(\phi^{\circ}_t)+\sum_{2\nmid t,0\leq t\leq k}\chi(\phi^{\circ}_t).
\end{split}
\]
After comparing with equation \eqref{Hpv}, we have
\begin{equation}\label{almostthere}
\begin{split}
\chi(Y)&=\sum_{2|t,0\leq t\leq k}2\chi(\phi^{\circ}_t)+\sum_{2\nmid t,0\leq t\leq k}\chi(\phi^{\circ}_t)\\
&=2\chi(\phi^{\circ}_0)+\sum_{2|t,2\leq t\leq k}2\chi(\phi^{\circ}_t)+\sum_{2\nmid t,0\leq t\leq k}\chi(\phi^{\circ}_t).
\end{split}
\end{equation}
Let $\phi_{t,Z}$ and $\phi^{\circ}_{t,Z}$ be the preimages of  $\phi_{t,Z}$ and $\phi^{\circ}_{t,Z}$ respectively for $t=0,1,\ldots,k$. Since $\sigma:Z \rightarrow \mathbb{P}W$ is a double cover ramified at $\phi_{1Z}$,
we have $\chi(\phi^{\circ}_{0Z})=2\chi(\phi^{\circ}_0)$ and $\chi(\phi^{\circ}_{t,Z})=\chi(\phi^{\circ}_Z)$, $t=1,2,\ldots,k$.
By Remark \ref{cstofZ},  we have 
$$\chi_{cst}(Z) = \sum_{0\leq t \leq k} \chi(\phi^{\circ}_{t,Z}) S_{cst}(y\in \phi^{\circ}_{t,Z}).
$$
By Proposition \ref{gammaa} we know that 
\begin{equation}\label{gammab}
 S_{cst}(y\in \phi^{\circ}_{t,Z})= \chi_t=\begin{cases}
    1, & \text{if $1\leq t\leq k, t\text{ is odd}$};\\
    2, & \text{if $1\leq t\leq k, t\text{ is even}$}.
  \end{cases}
\end{equation}
Since $Z$ is smooth along $\phi_{0Z}$, we see that $S_{cst}(y\in \phi^{\circ}_{0Z})=1$.
Thus, we obtain
\[\begin{split}
\chi_{cst}(Z) & = \chi(\phi^{\circ}_{0Z})+\sum_{2|t,2\leq t\leq k}2\chi(\phi^{\circ}_{t,Z})+\sum_{2\nmid t,0\leq t\leq k}\chi(\phi^{\circ}_{t,Z})\\
&=2\chi(\phi^{\circ}_0)+\sum_{2|t,2\leq t\leq k}2\chi(\phi^{\circ}_t)+\sum_{2\nmid t,0\leq t\leq k}\chi(\phi^{\circ}_t).
\end{split}\]
It remains to use equation \eqref{almostthere} to get $\chi_{cst}(Z)=\chi(Y).$
\end{proof}

\section{Comments and open questions.}\label{open}
In this section we make several comments and highlight some natural open questions which will hopefully be soon addressed by us or other researchers.

\medskip
Let $\widetilde{X}$ be the blowup of $X$ with center $Y$, where $Y$ is a locally complete intersection subscheme of codimenion $c$. Theorem 1.6 in \cite{ku2} shows that the numbers of components in the semiorthogonal decomposition of $D^b(coh(\widetilde{X}))$ is related to the discrepancy $c-1$ of the blowup.
Thus we get the inspiration that the meaning of discrepancy changes we made to define Clifford-stringy  Euler characteristics are expected to be interpreted in terms of the Lefschetz decompositions of derived categories. However, we do not at the moment understand this mechanism.

\medskip

The main theorem of this paper can be conjecturally generalized to a statement on Hodge polynomials. That is to say it is natural to construct some kind of Hodge polynomial of the non-commutative Clifford algebra $(\mathbb{P}(L),\mathcal{B}_0)$ which are predicted to be the same as the Hodge polynomial of the complete intersection $X_L$. Also, we could extend the statement to elliptic genera. However, it is really strange that we need to change the discrepancies, since this breaks the Calabi-Yau condition which is typically needed for modularity. 
This indicates that perhaps there are additional subtleties that come up when one tries to define elliptic genus for sigma models with non-commutative targets.

\medskip

It is natural to try to extend our calculations to the case of Clifford double mirrors of complete intersections in toric varieties, described in \cite{LBoZL}. Specifically, under some centrality and flatness assumptions, Theorem 6.3 of \cite{LBoZL} shows that a complete intersection $\Xx$ and a Clifford
non-commutative variety $\Yy$  given by different decompositions of the
degree element $\deg^\vee$ of a reflexive Gorenstein cone $K^\vee$
and the appropriate regular simplicial fans in $K^\vee$ have equivalent 
bounded derived categories. 
It is natural to conjecture that the Clifford-stringy Euler characteristics of $\Yy$ and the Euler characteristics of complete intersections $\Xx$ in toric varieties are equal. One would need to work with DM stacks, and to modify the definitions accordingly.

\section{Appendix: resolution of $z^2=x_1x_2 \cdots x_n$.}\label{res}

We describe in detail a resolution of the affine toric variety given by $z^2=x_1x_2 \cdots x_n$ for any $n$. We call this variety $\widehat{Z}_{toric}$.
The corresponding cone $\sigma$ in the cocharacter lattice $N$ is described as follows:
$$N=\{(a_1,\dots,a_n)\in \mathbb Z^l| \sum_{i=1}^l a_i\equiv 0{\rm (mod 2)}\}$$
the cone $\sigma$ is generated by the lattice elements $e_i=(0,\dots,0,2,0,\dots,0)\in N$ whose $i-$th position is $2$, where $i=1,\dots,l$.
The ray $\mathbb{R}_{\geq 0}e_i$ corresponds to a divisor which we denote by $E_{\text{2i-1,toric}}$.
The resolution of the singularities of $\widehat{Z}_{toric}$ is given by the subdivision of the cone $\sigma$. Any subdivision can be realized by a triangulation of the $(l-1)$-simplex spanned by $e_1,\dots,e_l$. We will give such a triangulation below.

\medskip

For a subset $\{i_1,\dots, i_\alpha\}\subseteq \{1,2,\dots,l\}$, $i_1<i_2<\dots <i_\alpha$, we denote the simplex spanned by $\{e_{i_1},\dots, e_{i_\alpha}\}$ to be $S(i_1,\dots, i_\alpha)$. We will give the triangulation $\widetilde{S}(i_1,\dots, i_\alpha)$ of $S(i_1,\dots, i_\alpha)$ inductively.
When $\alpha=2$, $S(i_1,i_2)$ is a segment, the triangulation $\widetilde{S}(e_1,e_2)$ is given by adding a vertex $e_{i_1,i_2}=\frac{e_{i_1}+e_{i_2}}{2}$.
Suppose we have given the triangulations of all simplices for $\alpha\le \beta$. Then for $\alpha=\beta+1$,
\begin{itemize}
\item first, we triangulate the simplices $S(i_1,\dots, i_\beta)$ and $S(i_2,\dots, i_{\beta+1})$;
\item second, we add a vertex $e_{i_1,i_{\beta+1}}=\frac{e_{i_1}+e_{i_{\beta+1}}}{2}$;
\item third, we add all the simplices generated by the simplices of the triangulation $\widetilde{S}(i_1,\dots, i_\beta)$ and the added vertex $e_{i_1,i_{\beta+1}}$, i.e. $Conv(\tau, e_{i_1,i_{\beta+1}})$ for all $\tau$ in the triangulation of $S(i_1,\dots, i_\beta)$.
    We also add all the simplics generated by the simplices of the triangulation $\widetilde{S}(i_2,\dots, i_{\beta+1})$ and the added vertex $e_{i_1,i_{\beta+1}}$,  i.e. $Conv(\tau, e_{i_1,i_{\beta+1}})$ for all $\tau$ in the triangulation of $S(i_2,\dots, i_{\beta+1})$.
\end{itemize}
This gives a triangulation of $S(i_1,\dots, i_\alpha)$. The subdivision of $\sigma$ is the fan over the triangulation $\widetilde{S}(1,2,3,\dots,l)$, which we denote by $\Sigma$.

\medskip

We can also describe the subdivision of $\sigma$ in another way.
Consider $l$ points on the real line  $1,2,3,\dots,l$. Let $P$ be the partially ordered set consisting of all closed segments with two endpoints among the $l$ points $1,2,3,\dots,l$, including zero length segments. We have a one-to-one correspondence between $P$ and the generators of rays of $\Sigma$.
\[\begin{split}
& P \overset{1:1}{\longleftrightarrow} \Sigma(1)=\{\text{generators of rays of }\Sigma\} \\
& [s,t] \longmapsto e_{[s,t]}=\frac{e_s+e_t}{2}
\end{split}\]
When $s=t$, we have $e_{s,s}=e_s$.
\begin{definition}
A subset $\Gamma\subseteq P$ is called a nested collection of segments if for any $\alpha, \beta \in \Gamma$, either $\alpha\subseteq\beta$ or $\beta\subseteq\alpha$. We denote the elements of $\Gamma$ in the decreasing order by $\Gamma=\{[s_\lambda,t_\lambda]|1\le s_\lambda\le t_\lambda\le l, \lambda=1,2,\dots,l'\}$, where $l'= \# \Gamma$. For all $\lambda< \nu$, we have $s_\lambda\leq s_\nu$ and $t_\lambda \geq s_\nu$.
\end{definition}
Then we have the following proposition.
\begin{proposition}
A subset of $P$ is a nested collection of segment if and only if the corresponding generators of rays of $\Sigma$ span a cone in $\Sigma$.
\end{proposition}

\begin{proof}
Left to the reader.
\end{proof}
\begin{remark}
If we regarded $e_{s_\lambda,t_\lambda}$ as a point in $\widetilde{S}(1,2,3,\dots,l)$, this proposition is equivalent to the statement that the set of points $\{e_{s_1,t_1},e_{s_2,t_2},\dots,e_{s_{l'},s_{l'}}\}$ corresponds to a simplex in $\widetilde{S}(1,2,3,\dots,l)$ if and only if $\Gamma$ is nested.
\end{remark}

In particular, a maximum dimensional cone in $\Sigma$ corresponds to a nested collection of segments $\Gamma=\{[s_\lambda,t_\lambda]|1\le s_\lambda\le t_\lambda\le l, \lambda=1,2,\dots,l\}$ such that the length of the segment $[s_\lambda,t_\lambda]$ is $l-\lambda$, i.e. the length of segments are $l-1, l-2, \ldots, 0$. Thus $\Gamma=\{[1,l], [s_2, t_2], \ldots, [s_l, t_l]\}$, where $t_l=s_l$ and $[s_{\lambda+1},t_{\lambda+1}]=[s_\lambda,t_\lambda-1] \text{ or } [s_\lambda+1,t_\lambda]$ for $\lambda=1,\dots, l-1.$

\begin{figure}[H]
  \centering
  \includegraphics[width=0.6\textwidth]{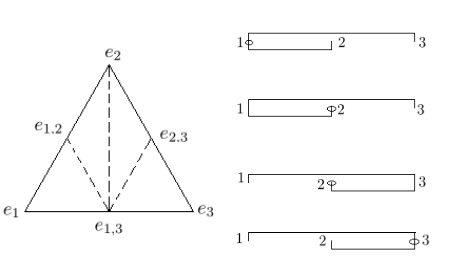}
  \caption{$\widetilde{S}(1,2,3)$, the triangulation of a triangle.}
  \label{fig:triangulation 1}
\end{figure}

\begin{figure}[H]
  \includegraphics[width=0.9\textwidth]{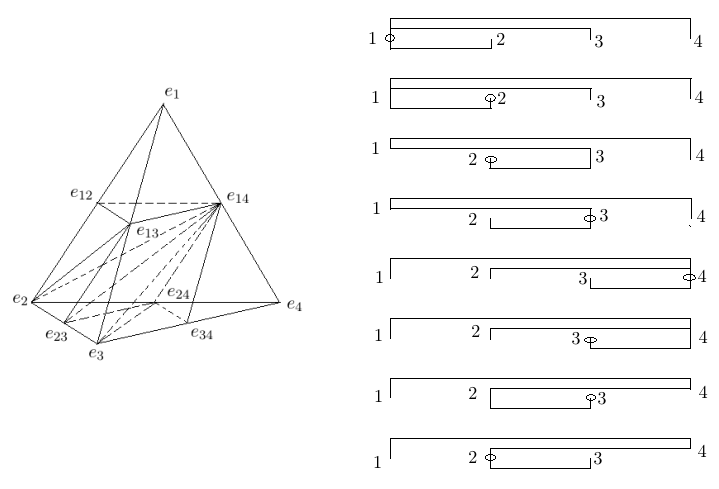}
  \caption{$\widetilde{S}(1,2,3,4)$, the triangulation of a tetrahedron.}
  \label{fig:triangulation 2}
\end{figure}
Please see Figure \ref{fig:triangulation 1} and Figure \ref{fig:triangulation 2} as examples for the triangulation of $S(1,2,3,\dots,l)$ in
case $l=3$ and $l=4$. We also describe the maximum simplices of the triangulation in terms of nested collections of segments.
\medskip

The fan $\Sigma$ corresponds to a toric variety which is denoted by $\widetilde{Z}_{toric}$.
We can see that subdivision $\Sigma$ consists of $2^{l-1}$ $l-$dimensional simplicial cones. Indeed, when constructing maximum simplices, there are two choices of elements of $P$ for each $\lambda=2,3,\ldots, l$.

\begin{proposition}
The variety $\widetilde{Z}_{toric}$ that corresponds to $\Sigma$ is smooth.
\end{proposition}

\begin{proof}
The multiplicity of each cone is $1$ in the lattice $N$. So $\widetilde{Z}_{toric}$ is smooth.
Thus $\Sigma$ gives a resolution of singularities of $\widehat{Z}_{toric}$, $\mu_{toric}: \widetilde{Z}_{toric}\to \widehat{Z}_{toric}$.
\end{proof}

In Section \ref{mo-dis-cli}, we will introduce a divisor which corresponds to a certain piecewise linear function on $\Sigma$. We describe this function below. While it is not logically necessary for our main theorem, we will show  that this function is (weakly) convex.

\begin{proposition}
Let $\varphi$ be the $\Sigma-$piecewise linear function $\varphi: \sigma\rightarrow \mathbb{R}$ which takes values
\begin{equation}\label{vertex}
  \varphi(e_{[s,t]})=\begin{cases}
    (2s-1)^2-1, & \text{if $s=t$ };\\
    \frac{(2s-1)^2+(2t-1)^2}{2}+3(t-s)-1 , & \text{if $s < t$}.
  \end{cases}
\end{equation}
Then $\varphi$ is locally convex, in the sense that $\varphi(\alpha x + \beta y)\geq \alpha \varphi(x)+ \beta \varphi(y)$ for $\alpha, \beta \geq 0$ and $x,y \in\sigma$.
\end{proposition}

\begin{proof}
For a maximum cone $\tau \in \Sigma$, we will introduce the notation $\varphi_{\tau}$ for the linear function on $\mathbb{R}^l$ that is equal to $\varphi$ on $\tau$. It is enough to prove that $\varphi$ is locally convex when restricted to a generic \footnote{ Generic means the segment is disjoint from codimension 2 cones of $\Sigma$.} segment in $\sigma$. So it is sufficient to prove
that for any pair of adjacent $l-$ dimensional cones $\tau$ and $\tau'$ \footnote{Simplices $\tau$ and $\tau'$ are called adjacent iff they have a common codimension one face.}, we have $$\varphi_{\tau}(x) \leq \varphi_{\tau'}(x)$$ for all $x\in \tau$. Moreover, it is enough to prove $\varphi_{\tau}(x) \leq \varphi_{\tau'}(x)$ for the generator x of $\tau$ which is not in $\tau'$.

\smallskip
If two simplices $\tau$ and $\tau'$ are adjacent, the corresponding collections of nested segments $\{[s_\lambda,t_\lambda]|\lambda=1,2,\dots,l\}$ and $\{[s'_\lambda,t'_\lambda]|\lambda=1,2,\dots,l\}$ fall into the following two cases:

Case 1. For some $\nu\in \{2,\dots, l-1\}$, we have $[s_\lambda,t_\lambda]=[s'_\lambda,t'_\lambda]$ when $\lambda\ne \nu$. After a possible switch of $\tau$ and $\tau'$, we denote
 \[\begin{split}
 &[s_{\nu-1},t_{\nu-1}]=[s,t]=[s'_{\nu-1},t'_{\nu-1}] \text{ for some } s,t \in \{1,2,\ldots,l\},\\
 &[s_\nu,t_\nu]=[s,t-1]\neq [s+1,t]=[s'_\nu,t'_\nu],\\
 &[s_{\nu+1},t_{\nu+1}]=[s+1,t-1]=[s'_{\nu+1},t'_{\nu+1}].
 \end{split}\]

Case 2. We have $[s_\lambda,t_\lambda]=[s'_\lambda,t'_\lambda]$ when $\lambda\ne l$. After a possible switch of $\tau$ and $\tau'$, we denote
\[\begin{split}
&[s_{l-1},t_{l-1}]=[s,s+1]=[s'_{l-1},t'_{l-1}],\\
&[s_l,t_l]=[s,s]\neq [s+1,s+1]=[s'_l,t'_l].
\end{split}\]

\smallskip
We first consider Case 1.
We know that $e_{[s,t-1]}$ is the only generator of $\tau$ that is not in $\tau'$. Since $e_{[s,t]}+e_{[s+1,t-1]}=e_{[s,t-1]}+e_{[s+1,t]}$, we have
\[\begin{split}
&\varphi_{\tau'}(e_{[s,t-1]})=\varphi_{\tau'}(e_{[s,t]}+e_{[s+1,t-1]}-e_{[s+1,t]})\\
&=\varphi_{\tau'}(e_{[s,t]})+\varphi_{\tau'}(e_{[s+1,t-1]})-\varphi_{\tau'}(e_{[s+1,t]})\\
&=(\frac{(2s-1)^2+(2t-1)^2}{2}+3(t-s)-1)\\
&+(\frac{(2(s+1)-1)^2+(2(t-1)-1)^2}{2}+3(t-s-2)-1)\\
&-(\frac{(2(s+1)-1)^2+(2t-1)^2}{2}+3(t-s-1)-1)\\
&=\frac{(2s-1)^2+(2(t-1)-1)^2}{2}+3(t-1-s)-1=\varphi_{\tau}(e_{[s,t-1]}).
\end{split}\]
So in this case, we have $\varphi_{\tau}=\varphi_{\tau'}$.

\medskip

Then, we consider Case 2.
We know that $e_s$ is the only generator of $\tau$ that is not in $\tau'$. Since $2e_{[s,s+1]}=e_s+e_{s+1}$, we obtain
\[\begin{split}
&\varphi_{\tau'}(e_s)=\varphi_{\tau'}(2e_{[s,s+1]}-e_{s+1})\\
&=2\varphi_{\tau'}(e_{[s,s+1]})-\varphi_{\tau'}(e_{s+1})\\
&=2(\frac{(2s-1)^2+(2(s+1)-1)^2+3-1}{2})-((2(s+1)-1)^2-1)\\
&=(2s-1)^2+3>(2s-1)^2-1=\varphi_{\tau}(e_s).
\end{split}\]
So in this case, we have $\varphi_{\tau}(x) \leq \varphi_{\tau'}(x)$ for the generator $x$ of $\tau$ which is not in $\tau'$.

Thus we finish the proof.
\end{proof}

\begin{remark}
In fact, Case $1$ of the above argument shows that the $\varphi_\tau$ depends only on the $0-$length segment in $\tau$. As a consequence, for two maximum cones $\tau_1,\tau_2 \in \Sigma$, we have $\varphi_{\tau_1}=\varphi_{\tau_2}$ if and only if $\tau_1$ and $\tau_2$ contain the same vertices of $\sigma$.
\end{remark}

\medskip

We will now describe how one can find recursively a certain quantity involved in the calculation of Clifford-stringy Euler characteristics which we use in Section \ref{contribu}.
\begin{proposition}\label{GrecPhi}
Let $T$ be a subset of $\{1,\ldots,n\}$. Define $G_\varphi[T]$ to be the
sum over the maximum cones $C$ of the fan induced by $\Sigma$ on the span of $\{e_i,i\in T\}$
$$
G_\varphi[T]=
\sum_{C} \prod_{e_{[s,t]}\in C} \frac 1{1+\varphi(e_{[s,t]}) }.
$$
Then $G_\phi[T]$ satisfies the following recursive relations.
\[
\begin{split}
&G_\varphi[\{s\}]=\frac {1}{1+\varphi(e_{[s,s]})} =\frac{1}{(2s-1)^2}, \text { and }\\
&G_\varphi[T]=\frac{1}{1+\varphi(e_{[\min(T),\max(T)]})}
(G_\varphi[T-\max(T)]+G_\varphi[T-\min(T)]).
\end{split}\]
\end{proposition}

\begin{proof}
The $| T|=1$ case is clear. 
For $|T|>1$ the maximum cones of $\Sigma$  restricted to the span of $\{e_{i},i\in T\}$  correspond  to nested segments with ends in $T$. Thus, they
all contain 
$$
e_{[\min(T),\max(T)]}.
$$
Moreover, these cones are either spanned by the above vertex and the cones for $T-\min T$ or by
the above vertex and the cones for $T-\max T$.
\end{proof}

\end{document}